\documentclass[11pt, oneside]{article}  
\usepackage{geometry}                		
\usepackage[pagewise]{lineno}	
\usepackage[parfill]{parskip}    		
\usepackage{graphicx}				
\usepackage[utf8]{inputenc}
\usepackage[english]{babel}
\usepackage{mathrsfs,amsmath}
\usepackage{amsthm}								
\usepackage{amssymb}
\usepackage{amsmath}
\usepackage{amsfonts}
\usepackage{mathbbol}
\usepackage{biblatex}
\usepackage{xpatch}
\usepackage{subcaption}
\usepackage{hyperref}
\usepackage{csquotes}

\newtheorem{theorem}{Theorem}[section]
\makeatletter
\AtBeginDocument{\xpatchcmd{\@thm}{\thm@headpunct{.}}{\thm@headpunct{}}{}{}}
\makeatother
\newtheorem{lemma}[theorem]{Lemma}

\newtheorem{proposition}[theorem]{Proposition}
\newtheorem{example}[theorem]{Example}

\theoremstyle{definition}
\newtheorem{definition}{Definition}[section]

\theoremstyle{remark}
\newtheorem{remark}{Remark}[section]

\author{Wei Lin}
\date{}
\title{On Crossing Ball Structure in Knot and Link Complements}	

\begin{document}
\maketitle
\begin{abstract}
We develop a word mechanism applied in knot and link diagrams for the illustration of a diagrammatic property. We also give a necessary condition for determining incompressible and pairwise incompressible surfaces, that are embedded in knot or link complements. Finally, we give a finiteness theorem and an upper bound on the Euler characteristic of such surfaces.
\end{abstract}

\section{Introduction} 

\subsection{Preliminary Discussion}
Let $ L \subset S^3$ be a link and $\pi(L) \subset S^2 (\subset S^3)$ be a regular link projection.  Additionally, let $F \subset S^3 - L$ be an closed incompressible surface.
In 1981 Menasco introduced his crossing ball technology for classical link projections \hyperref[B7]{[7]} that replaced $\pi(L)$ in $S^2$ with two 2-spheres, $S^2_\pm$, which had the salient features that $L$ was embedded in $S^2_+ \cup S^2_-$ and $S^3 \setminus (S^2_+ \cup S^2_-)$ was a collection of open $3$-balls---$B^3_\pm$ that correspond to the boundaries $S^2_\pm $ and a collect of {\em crossing balls}.  (Please see \S  \ref{1.2} for formal definition.)  Using general position arguments, this technology allows for placing $F$ into {\em normal position} with respect to $S^2_\pm$ so that $F \cap S^2_\pm$ is a collection of simple closed curves (s.c.c.'s).  When one imposes the assumption that $\pi(L)$ is an alternating projection, the normal position of an essential surface is exceedingly well behaved to the point where by direct observation one can definitively state whether the link is split, prime, cabled or a satellite.  As such, any alternating knot can by direct observation be placed into one of William Thurston's three categories---torus knot, satellite knot or hyperbolic knot \hyperref[B10]{[10]}.

One salient result from \hyperref[B7]{[7]} is that any essential surface in a non-split alternating link exterior will contain a meridional curve of a link component and, thus, studying such essential surfaces can be reduced to studying essential surfaces with meridional boundary curves that are meridianally incompressible or {\em pairwise incompressible}.  The importance of studying pairwise incompressible surfaces has been reflected in the work of numerous scholars. To name a few, Bonahon and Seibenmann's work on arborescent knots \hyperref[B3]{[3]}, Oertel's work on star links \hyperref[B9]{[9]}, Adams' work on toroidally alternating links \hyperref[B1]{[1]}, Adams' et. al work on almost alternating links \hyperref[B2]{[2]}, Fa's initial cataloging of incompressible pairwise incompressible patterns \hyperref[B4]{[4]}, Lozanoand-Przytycki work on $3$-braid links \hyperref[B6]{[6]}, and Hass-Thompson-Tsvietkova results on growth of the number of essential surfaces in alternating link complements \hyperref[B5]{[5]}.  However, to-date there has not been a criterion for determining whether a closed incompressible surface in an arbitrary link complement, when presented in Menasco's normal form, is pairwise incompressible.  The main result of this note partially fills this gap in the literature by giving such a necessary condition for determining pairwise compressibility/incompressibility---Theorem \ref{Main}.

For reader familiar with the arguments in \hyperref[B7]{[7]}, our approach to proving Theorem \ref{Main} should be of interest.  The techniques \hyperref[B7]{[7]} allow for multiple ways for arguing the existence of a meridional curve.  The most direct way is given in the proof of Theorem 2 (The Meridian Lemma) of \hyperref[B7]{[7]}.  In brief, once a surface $F$ is in normal position, for an alternating projection the existence of such a meridional curve is manifested by an innermost s.c.c. on $S^2_\pm$ of $F \cap S^2_\pm$.  An alternative indirect way of arguing comes from the proof of Lemma 2 of \hyperref[B7]{[7]}.  Again in brief, for $F$ in normal position one considers an arbitrary s.c.c. $C \subset F \cap S^2_\pm$.  Then when one considers how $C {\rm 's}$ imposes a ``nesting behavior'' on the subset of curves in $F \cap S^2_\mp$ that intersect $C$, the assumption of an alternating projection again forces the existence of some s.c.c. from this subset manifesting a meridian.

This latter technique for detecting the existence of a meridional curve is the main tool employed in this note. Our concluding remark is that it seems to be an underutilized tool to-date in works that have exported the crossing ball technology to study pairwise incompressibility in other settings.\par

\begin{figure}[h]
\begin{subfigure}{.49\textwidth}
  \centering
  \includegraphics[width=.8\linewidth]{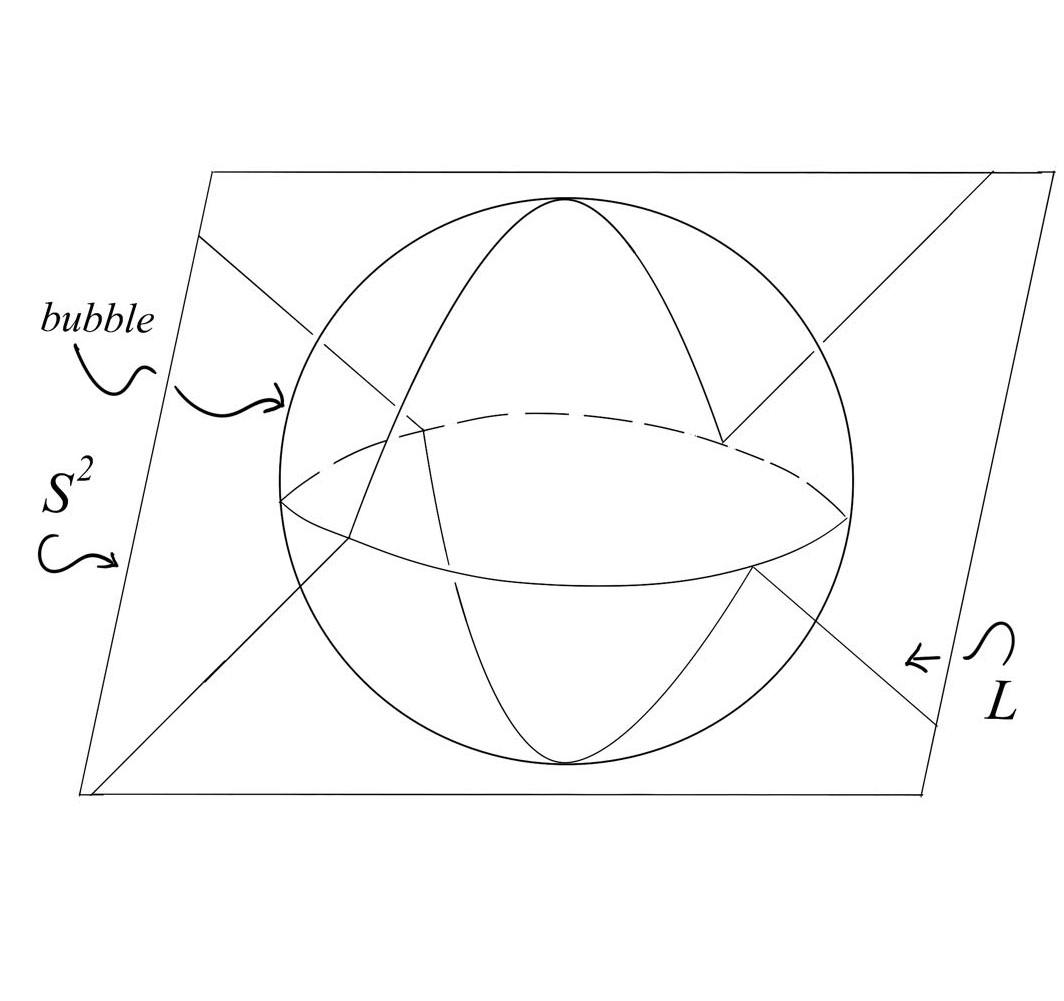}
  \caption{The bubble.}
  \label{the bubble}
\end{subfigure}
\begin{subfigure}{.49\textwidth}
  \centering
  \includegraphics[width=.8\linewidth]{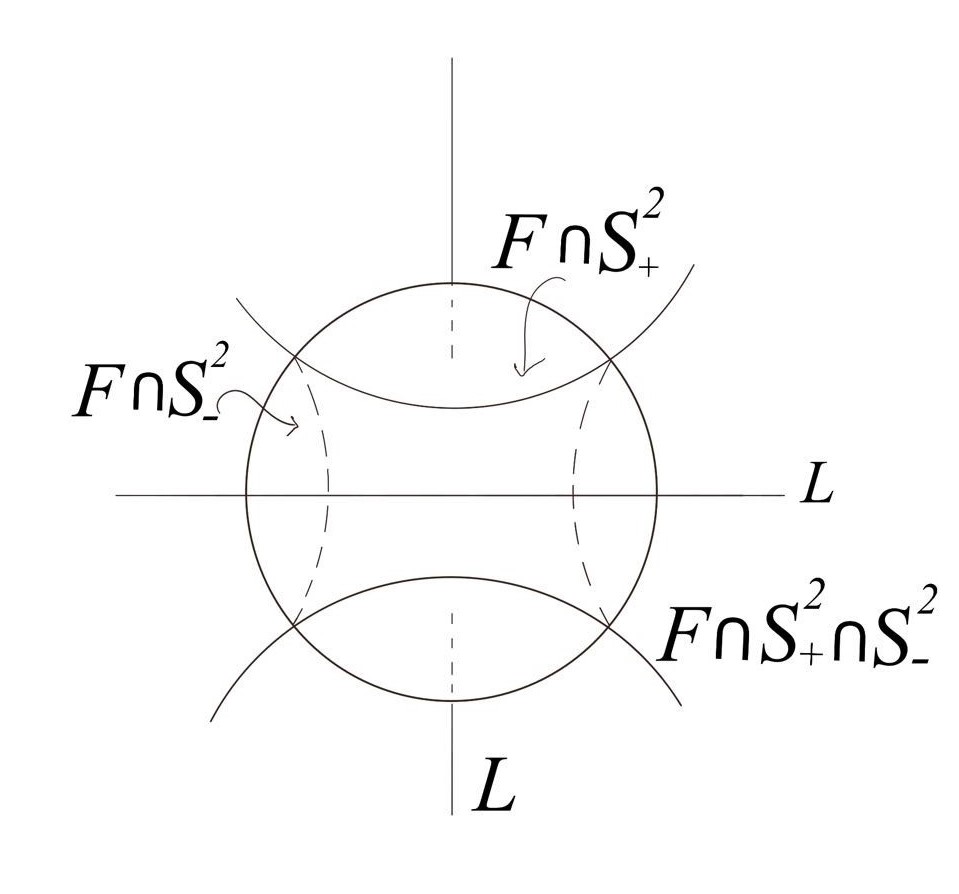}
  \caption{A local view of the crossing ball with a saddle. The solid arcs on the crossing ball are part of $F \cap S^{2}_{+}$, while the dotted arcs on the crossing ball are part of $F \cap S^{2}_{-}$.}
  \label{F0}
\end{subfigure}
\caption{}
\label{Bubble}

\end{figure}

\subsection{Main Results}\label{1.2}

To review, a surface $F$ properly embedded in the link $L$ complement in $S^3$ is called \emph{pairwise incompressible} if for each disk $D$ in $S^3$ meeting the $L$ transversely in one point, with $D\cap F=\partial D$, there is a disk $D'\subset F\cup L$ meeting $L$ transversely in one point, with $\partial D'= \partial D$. We consider the following problem: Given the condition that a surface is incompressible in the link complements, how to characterize such a surface that is also pairwise incompressible? For the rest of this paper, if a surface $F$ is connected, and it is both incompressible and pairwise incompressible, we say that $F$ is \emph{essential}, or $F$ is an \emph{essential surface}.


As before, let $\pi(L)$ be a connected regular projection from a knot or a link $L \subset S^{3}$ to a sphere $S^{2}$. We place a ball at each crossing of $\pi(L)$, which we refer to as a bubble $B$ (See Figure \ref{the bubble}). At each crossing, both the overstrand and the understrand are in the $\partial B$. We define $S^{2}_{+}$ to be the sphere $S^{2}$ where the equatorial disk in each bubble is replaced by the upper hemisphere of the bubble, and $B^{3}_{+}$ to be the 3-ball bounded by $S^{2}_{+}$ that does not contain any bubble $B$. Similarly, we can define $S^{2}_{-}$ and $B^{3}_{-}$ when replacing equatorial disk in each bubble by lower hemisphere of the bubble. Since such operations can be performed on an arbitrary link diagram on $S^{2}$, we use this convention that all $\pi(L)$ of this note are assumed to be endowed with the crossing ball structure.\par

Let $F\subset S^{3}-L$ be a separating sphere, a closed essential surface, or an essential surface with meridional boundaries (when $L$ is a knot, apparently we do not take the separating sphere into account). $F$ can be isotoped to intersect each bubble in a set of saddles. We assume the surface $F$ is connected and is chosen to minimize the total number of saddles and curves in lexicographical order. We can replace $F$ (isotope when it is essential) with $F'$  such that it is in normal position (details are shown in \S \ref{standard}). In general, when $F$ is essential, we can restrain $F$, so it is in such a position, that for each $C \subset F \cap S^{2}_{\pm}$: $C$ bounds a disk in either $B^{3}_{\pm}$; $C$ does not pass through a bubble twice; and, all meridional boundaries of $F$, i.e. \emph{punctures}, are on the arcs of $F\cap S^{2}_{+} \cap S^{2}_{-}$ and away from the crossing balls. As one stands inside of $B^{3}_{+}$ and look at  $S^{2}_{+}$, one should be able to see the graph of intersection $F \cap S^{2}_{+}$. We call this situation \emph{from $B^{3}_{+}$'s side of view} (or respectively, \emph{from $B^{3}_{-}$'s side of view}, if one stands in $B^{3}_{-}$).\par

For a loop $C \subset F \cap S^{2}_\pm$ we consider it the union of two types of arcs: type $I$ $A_{1} \subset C \cap S^2_+ \cap S^2_-$; and, type $II$ $A_2 \subset C \setminus C \cap S^2_+ \cap S^2_-$. (Notice an arc of type $I$ is away from bubbles, while an arc of type $II$ is on the boundary of a bubble and away from punctures). Arcs of type $I$ whose ends encounter two crossing balls that both lie on same side of $C$, are assigned a label $R$, while arcs of type $I$ whose ends encounter two crossing balls that lie on different sides of $C$, are assigned a label $\emptyset$, which will be omitted. Additionally, arcs of type $I$ that have intersections with $L$ are also assigned a label $P^{i}$, where $i$ is the number of times it intersects with $L$. Arcs of type $II$ are just assigned with a label $S$.\par

With the above-mentioned concepts, we now introduce the word mechanism. We consider the nontrivial situation when $C$ has nonempty intersections with saddles and give the following definition:\\ \\

\begin{figure}[h] 
\centering
\includegraphics[width=0.5\textwidth]{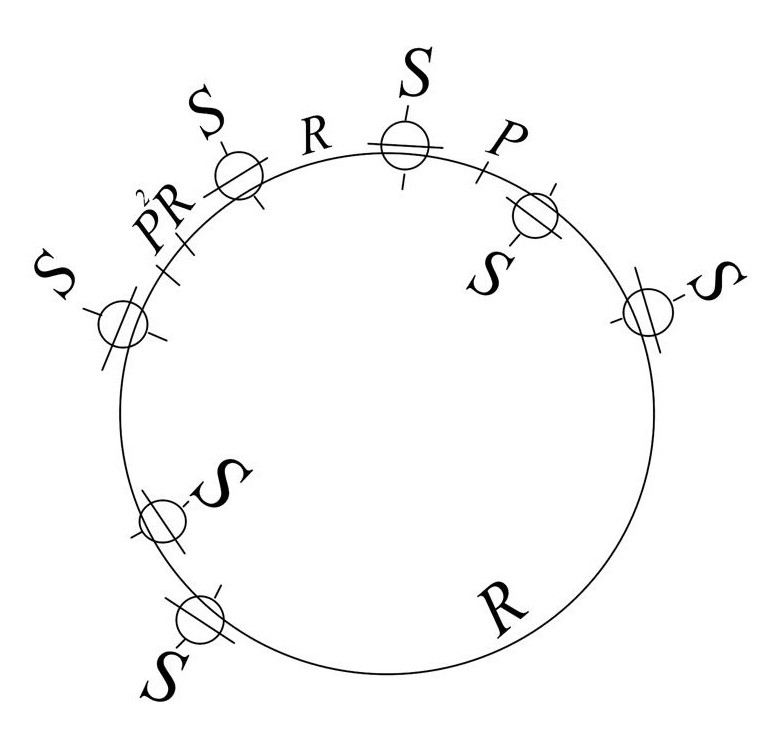}
\caption[A cyclic word.]{Start with any label, and choose an orientation. The labels are recorded as a cyclic word that can be read as $SP^{2}RSRSPSSRSS$. To be specific, we use the $i$-th power of $P$, $P^{i}$  to represent the surface is punctured $i$ times on the type-$II$ arc. $R$ is a label, that $SRS$ or $SP^{i}RS$ stands for successive saddles lie on the same side of the simple closed curve, regardless of the possible puncture(s) $P^i$ in between them. $SS$ or $SP^{i}S$ means the successive saddles lie on different sides of the simple closed curve (since the label $\emptyset$ of type $II$ arc is omitted), regardless of the $P^i$ in between them. The choices of starting label, orientation to record, or the order of $P$'s and $R$ on the same type $I$ arc, do not make a cyclic word different.}
\centering

\label{cyclic word}
\end{figure}

\theoremstyle{definition}
\begin{definition}[cyclic word]\label{def1}

 A \emph{cyclic word} $\omega(C)$ is a word obtained by recording in order the labels of the arcs of $C$. See Fig. \ref{cyclic word}.\par

\end{definition}

From a cyclic word $\omega(C)$ and the number of loops at each bubble intersects, we will be able to obtain a \emph{virtual word} $\omega^v(C)$ (see Definition  \ref{bubble}). We will show that each virtual word related to the essential surfaces in normal position satisfies a condition defined in \S \ref{C3} as \emph{$\omega$-reducible}. The diagrammatic property mentioned at the beginning of this note can be illustrated by the following two theorems:\par

\begin{theorem} \label{Main} Let $L \subset S^{3}$ be a link, $F \subset S^{3}-L$ be a separating sphere or an essential surface that is in normal position, then $\omega^{v}(C)$ is $\omega$-reducible for each simple closed curve $C \subset F \cap S^{2}_{\pm}$. \par
\end{theorem}

If in the link complement, there exists a separating sphere or a closed essential surface $F\subset S^{3}-L$, we also give a theorem to show that the configurations of such surfaces are finite, and we can give an upper bound on the Euler characteristic of such surfaces through the \emph{pullback graph} (see \S \ref{chapter final}) of $F$, i.e. the surface $F$ endowed with a 4-valent graph structure on which some of the edges are labeled with $R$. In the following theorem, $|R|$ stands for the total number of type $I$ arcs in $F\cap S^{2}_{+}$, or equivalently, in $F\cap S^{2}_{-}$, marked with $R$. And $n$-gons correspond to disks of $F\cap B^{3}_{\pm}$ which intersect with $n$ saddles:\par

\begin{theorem}\label{characterization}

Let $L\subset S^{3}$ be a link, $F\subset S^{3}-L$ be a separating sphere, or a closed essential surface in normal position. Then:\par
(a) If $|R|=4$ or $6$, $F$ is a sphere. If $|R|=8$, $F$ is either a sphere or a torus.\par
(b) The maximum vertex number of a region in the pullback graph is bounded by $|R|-2$.\par
(c) For fixed $|R|$, there are only finitely many such surfaces $F$ (up to isotopy when $F$ is essential).\par
(d) The Euler characteristics $\chi(F)$ of $F$ and the number of $n$-gons $F_n$ subject to the following restriction :

\begin{equation*}
\chi(F)=\sum_{n=2}^{|R|-2}F_{n}-\sum_{n=2}^{|R|-2}\frac{n}{4}F_{n}=\sum_{n=2}^{|R|-2}F_{n}-|S| \leq |R|-|S|\\ \\
\end{equation*}

\end{theorem}

\section{Normal Position and Motivating Examples}

\subsection{Normal position}\label{standard}

Suppose $F \subset S^{3}-L$ is a closed surface, or a surface whose boundary curves are all meridians of $L$ which do not intersect the bubbles. To each component $C$ of $F \cap S^{2}_{\pm}$ can be associated a cyclic word. Note that Definition  \ref{def1}  modifies the original representation in \hyperref[B7]{[7]}. When $\pi(L)$ is alternating, $SP^{2i+1}S$ of \hyperref[B7]{[7]} will be denoted as $SP^{2i+1}RS$ in this note; If the puncture number is even, the notation of this note is $SP^{2i}S$, according to the alternatingness of link diagram. \par

\begin{proposition} \label{sp1}

Let $F\subset S^{3}-L$ be a separating sphere or a closed essential surface, then $F$ can be replaced by another surface $F^{'}$ of the same type (isotopic to $F$ when it is closed incompressible pairwise incompressible) that is in the following position:\par
$(1)$ No word $\omega (C)$ associated to $F$ is empty.\par
$(2)$ No loop of $F \cap S^{2}_{\pm}$ meets a bubble in more than one arc.\par
$(3)$ Each loop of $F \cap S^{2}_{\pm}$ bounds a disk in $B^{3}_{\pm}$.\par

\end{proposition}

See \hyperref[B7]{[7]} Lemma 1 for proof. Notably it is independent of the alternatingness of $\pi(L)$. We say $F$ is in \emph{normal position of closed surface} if it satisfies all above conditions. Conditions (1) and (3) are the results of the incompressibility of $F$ and are independent of the alternatingess of $L$. Condition (2) is the result of both the pairwise incompressibility of $F$ and the choice of the isotopic class of $F$ to minimize saddles number, which are both independent on the alternating property of $\pi (L)$. A loop meeting both sides of a bubble will manifest a meridian curve, while a loop meeting the same side of a bubble will violate the choice of minimal saddles.\par

In addition to (1), (2) and (3) conditions for closed surface in normal position, we claim $F$ can be isotoped to satisfy additional three conditions if it has meridional boundaries:\par

\begin{proposition}\label{sp2}
Let $F\subset S^{3}-L$ be an essential surface with meridional boundaries, then $F$ can be isotoped so that, in additional to (1), (2) and (3), $F$ satisfies the following conditions:\par
(4) No loop of $F \cap S^{2}_{\pm}$ meets both a bubble and an arc of $L \cap S^{2}_{+} \cap S^{2}_{-}$ having an endpoint on that bubble.\par
(5) No loop of $F \cap S^{2}_{\pm}$ meets a component of $L \cap S^{2}_{\pm}$ more than once.\par
(6) There does not exist two loops $\alpha \subset F \cap S^{2}_{+}$ and $\beta \subset F \cap S^{2}_{-}$, with arcs $a,b \subset \alpha \cap \beta$ such that the interiors of $a,b$ are contained in adjacent components of $S^{2}_{+} \cap S^{2}_{-}-L$, and $\partial a \cap \partial b = \emptyset$.\par

\end{proposition}
See \hyperref[B8]{[8]} Lemma 3 for proof, as these properties are independent of the alternatingness of $L$. We say a surface with meridonal boundaries is in \emph{normal position} when it satisfies all (1)-(6) of the above conditions.\\ \\

\subsection{Motivating examples}

\begin{figure}[h]
\begin{subfigure}{.49\textwidth}
  \centering
  \includegraphics[width=.9\linewidth]{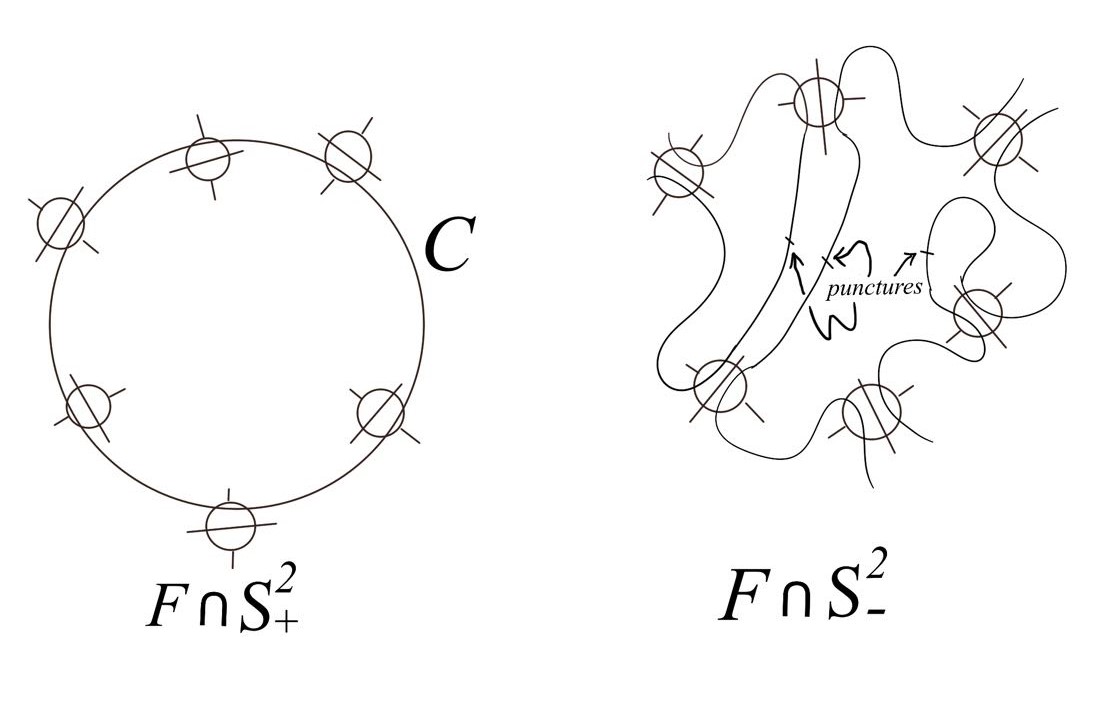}
  \caption{}
  \label{ex1a}
\end{subfigure}
\begin{subfigure}{.49\textwidth}
  \centering
  \includegraphics[width=.9\linewidth]{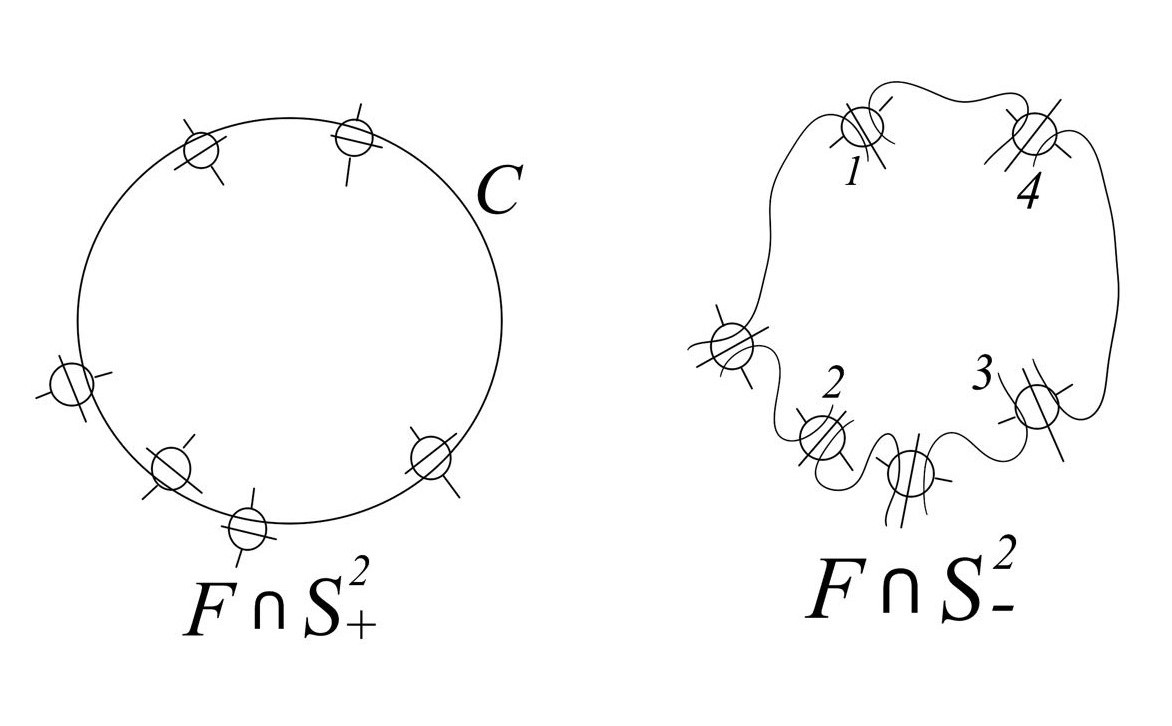}
  \caption{}
  \label{ex1b}
\end{subfigure}
\caption[Motivating example 1.]{}
\label{ex1}

\end{figure}

Let $F\subset S^{3}-L$ be a surface that is closed or with meridional boundaries. Suppose $F$ satisfies condition (1), (3), (4), (5), (6) of normal position, and $C \subset F \cap S^{2}_{\pm}$. We show a few examples which will explain the motive for some technical definitions in next section.

\begin{example}
Fig. \ref{ex1a} shows the situation $C$ intersects with saddles alternatingly. Assume $F$ is punctured, no matter how we ``connect'' the saddles on the same side of $C$ in $F\cap S^{2}_{+}$ through disk(s), on $F\cap S^{2}_{-}$ there is always a curve going through both sides of a bubble, contradicting the pairwise incompressibility of $F$ (see \hyperref[B1]{[1]}, the proof of Lemma 1). \par In Fig. \ref{ex1b}, $F$ is closed and $C$ no longer intersects saddles alternatingly. However, as one tries to connect the 1st arc end with the 2nd, 3rd, or 4th arc end, the resolving diagram represents the result of different ``connection of saddles'' through disk in $F\cap S^{2}_{+}$, and the resolving surface always contains a meridian curve.\par

\end{example}
 A natural question raised here is: How far away the link diagram is from alternating can we put a closed essential surface in the link complement? In order to answer this question, we will define such ``connection of saddles'' in the next section (see Definition  \ref{connection}).\par

\begin{example}
Fig. \ref{ex2} shows two examples when multiple saddles exist in a single bubble. In Fig. \ref{ex2a}, $F\cap S^{2}_{-}$, an arc goes through the same side of a bubble twice; While in Fig. \ref{ex2b} the surface $F$ is placed in normal position.\par

\end{example}

\begin{figure}[h]
\begin{subfigure}{.49\textwidth}
  \centering
  \includegraphics[width=.9\linewidth]{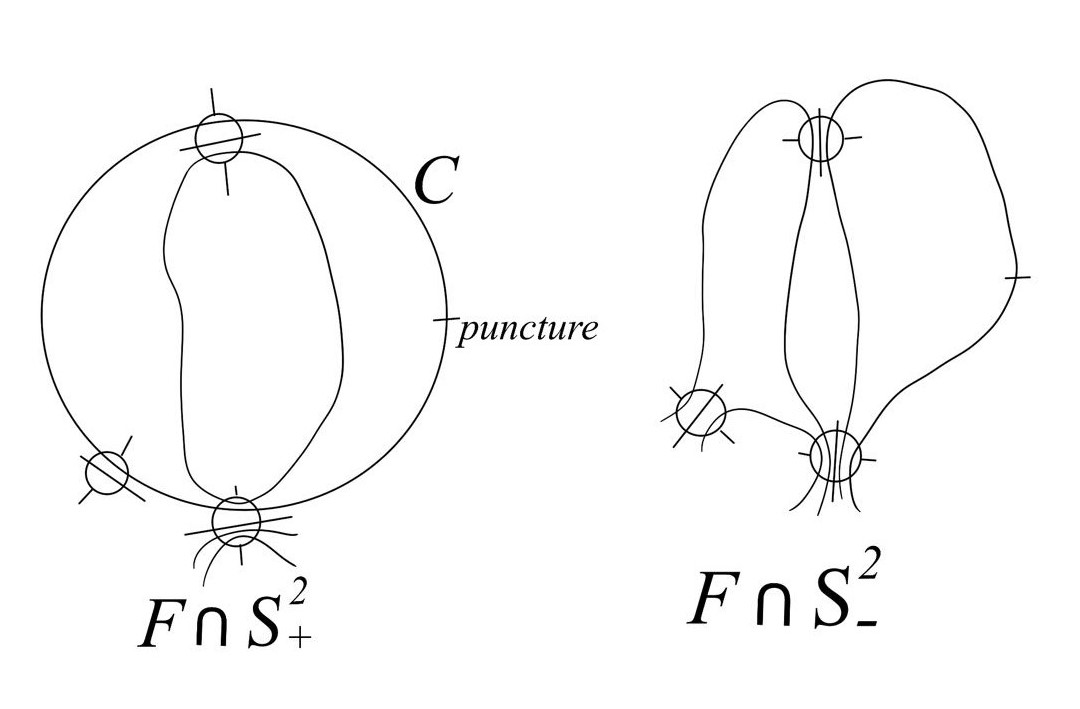}
  \caption{}
  \label{ex2a}
\end{subfigure}
\begin{subfigure}{.49\textwidth}
  \centering
  \includegraphics[width=.9\linewidth]{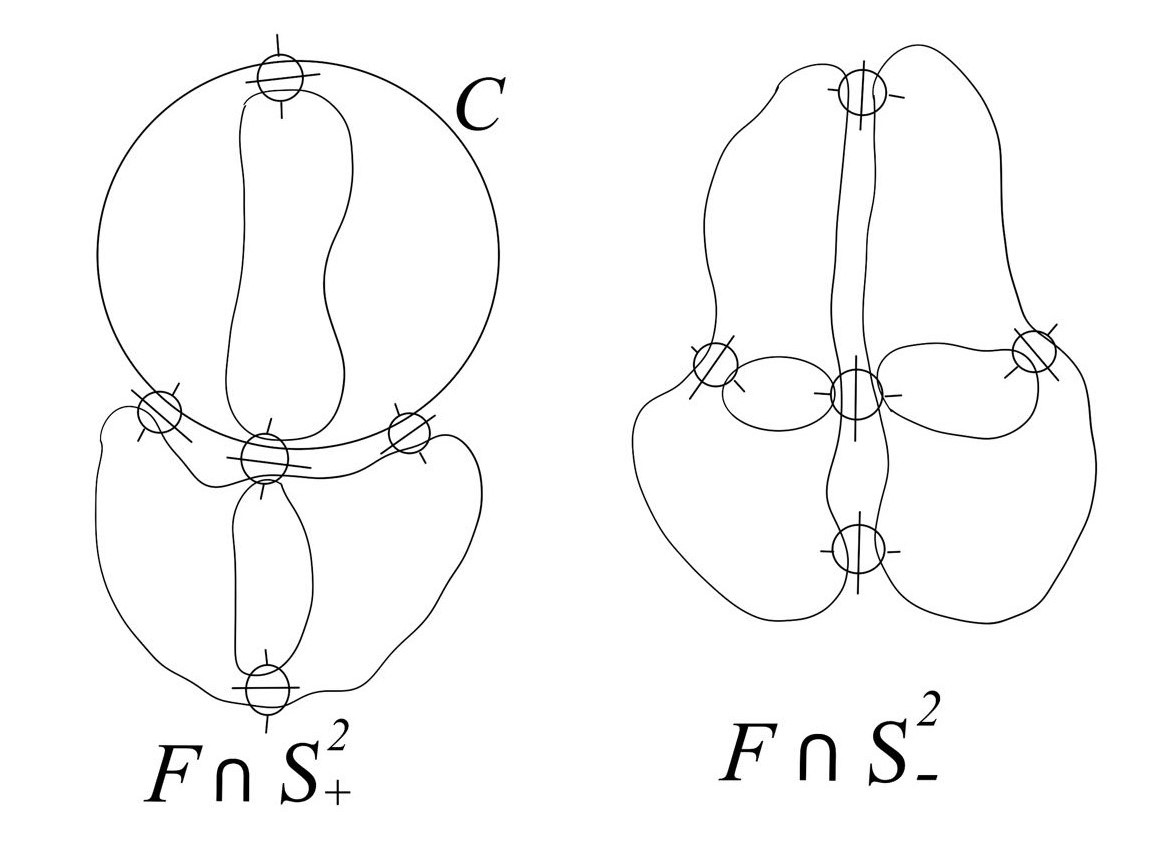}
  \caption{}
  \label{ex2b}
\end{subfigure}
\caption[Motivating example 2.]{}
\label{ex2}

\end{figure}

In order to illustrate how certain arrangement of saddles can lead to the difference between Fig. \ref{ex2a} and Fig. \ref{ex2b}, we give Definition  \ref{bubble} to keep record of the number of multiple saddles in a single bubble.\\ \\

\section{Word Reduction and Virtual Diagram}\label{C3}

\subsection{Cyclic Word and Word Reduction of Paired Up Saddles}
\theoremstyle{definition}
\begin{definition}[$\omega$-reducible]\label{def2}

We say a cyclic word is \emph{$\omega$-reducible} if it can be reduced to $\emptyset$ after we perform finitely many times of the following operations:\par
$($\RN{1}$)$ $R^{2i}\to \emptyset$, \par
$($\RN{2}$)$ $S^{i}RS^{i}\to R$,  \par
$($\RN{3}$)$ $(SR)^{2i}S\to R$,\par
$($\RN{4}$)$ $P^{i} \to \emptyset$,\par

where $i$ stands for a positive integer, $S^{i}$ means $i$ saddles in a sequence, and similarly for $R^{i}$, $P^{i}$. If a cyclic word is not $\omega$-reducible, we call it \emph{$\omega$-irreducible}.\par
\end{definition}

\begin{remark}

We call the above operations \emph{$\omega$-reductions} or \emph{reductions}. An $\omega$-reduction is a modification of the cyclic word. A cyclic word without $P$ must be even length, a $($\RN{1}$)$, $($\RN{2}$)$, or $($\RN{3}$)$ reduction does not change the parity of word length.\par
\end{remark}
We focus on closed surface and cyclic word with no puncture for convenience. The results of the case with punctures are going to be similar due to reduction $($\RN{4}$)$. Therefore we assume $F \subset S^{3}-L$ is a sphere, or a closed essential surface in normal position, $C \subset F \cap S^{2}_{+}$ (similarly we can give definitions for $C \subset F \cap S^{2}_{-}$). Then apparently $\omega(C)$ does not contain any $P$'s. With these assumptions, we give the following two technical definitions:\par

\theoremstyle{definition}
\begin{definition}[partial word]

A \emph{partial word} $\omega(C')$ is a record of an arc $C' \subset C$ that is a union of type $I$ or type $II$ arc(s), in order, the labels of these arc(s). We say a partial word of odd length is \emph{$R$-$\omega$-reducible} if it can be $R$ after we perform finitely many times of $($\RN{1}$)$, $($\RN{2}$)$, $($\RN{3}$)$ reductions. Otherwise, we call it \emph{$R$-$\omega$-irreducible}.\par 

\end{definition}

\begin{figure}[h]
\begin{subfigure}{.49\textwidth}
  \centering
  \includegraphics[width=.9\linewidth]{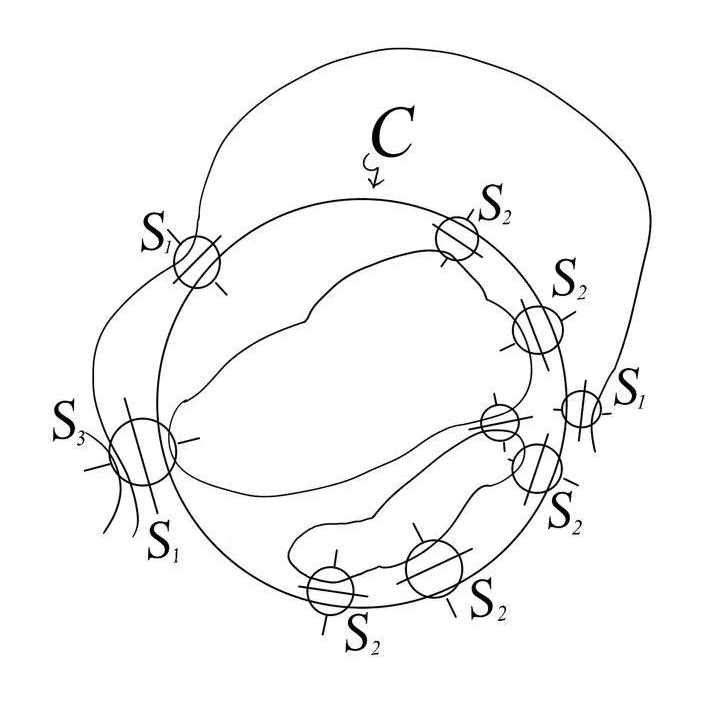}
  \caption{$F\cap S^{2}_{+}$}
 
\end{subfigure}
\begin{subfigure}{.49\textwidth}
  \centering
  \includegraphics[width=.9\linewidth]{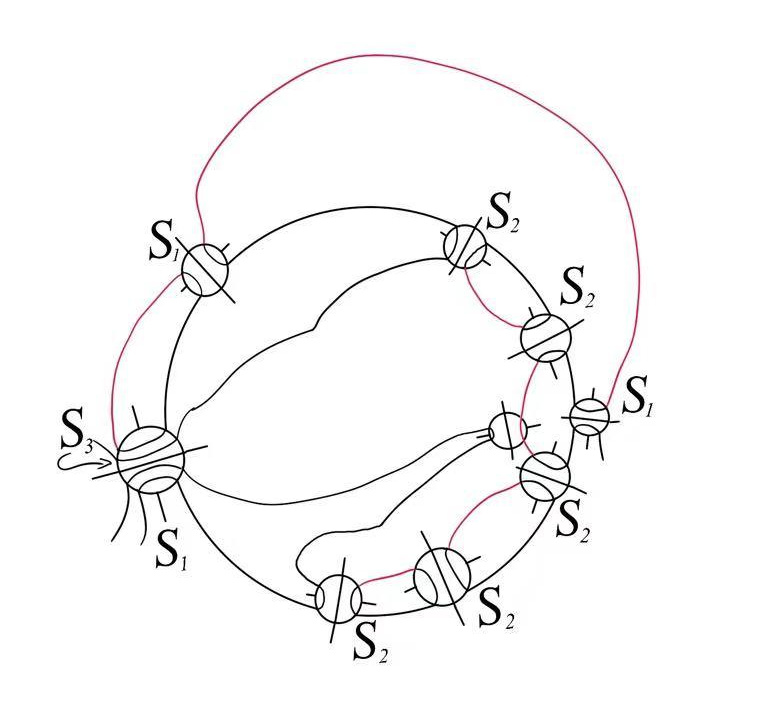}
  \caption{$F\cap S^{2}_{-}$}

\end{subfigure}
\caption[Paired up saddles.]{Three saddle $S_1$'s in the figure are paired up through a connecting operation $f_1$, similarly for the four $S_2$'s, because each red arc in (b) $F\cap S^{2}_{-}$ intersects paired up saddles at both ends. $S_3$ is not paired up with any $S_1$ or $S_2$, since $S_3$ does not intersect $C$. Any $S_1$ and $S_2$ are not paired up, since there does not exist an arc satisfying the definition. }
\label{paired}

\end{figure}

\theoremstyle{definition}
\begin{definition}[paired up saddles]\label{connection}

We denote the set of all saddles contained in the bubbles that $C$ intersects as $\Lambda_{C}$.
 Suppose two saddles that $C$ intersects both intersect with an arc $\mathscr{A} \subset F\cap S^{2}_{-}$, and $\mathscr{A} - \{C \cup \Lambda_{C} \} $ is a connected arc, then we say these saddles are assigned with a \emph{connecting operation} $f_{\lambda}$. We identify the notations of the connecting operations assigned to the same saddle, and we say a set of saddles are \emph{paired up} to each other(s), through $f_{\lambda}$, if all of them can be assigned with an $f_{\lambda}$. Moreover, we can mark the saddle $S$'s with the same subscripts of their assigned connecting operation, so they can be recorded as $S_{\lambda}$. See Fig. \ref{paired}. \par
\end{definition} 

\begin{remark}
This is well-defined since a saddle $C$ intersects can be assigned with no more than one connecting operation, after identification.\par
\end{remark}

With the above definitions, we can now show a lemma to describe the situation when each bubble that a selected loop $C$ intersects contains only one saddle. Then the purpose of rest of this section becomes transferring the other situations to the situation in the following lemma:\par

\begin{lemma} \label{lemma}
Let $L\subset S^{3}$ be a link, $F \subset S^{3}-L$ be a separating sphere or an essential surface in normal position.
If each saddle intersecting $C\subset F \cap S^{2}_{\pm}$ is paired up to some other saddle(s) intersecting $C$, then $\omega(C)$ is $\omega$-reducible. \par
\end{lemma}

\begin{figure}[h]
\begin{subfigure}{.49\textwidth}
  \centering
  \includegraphics[width=.8\linewidth]{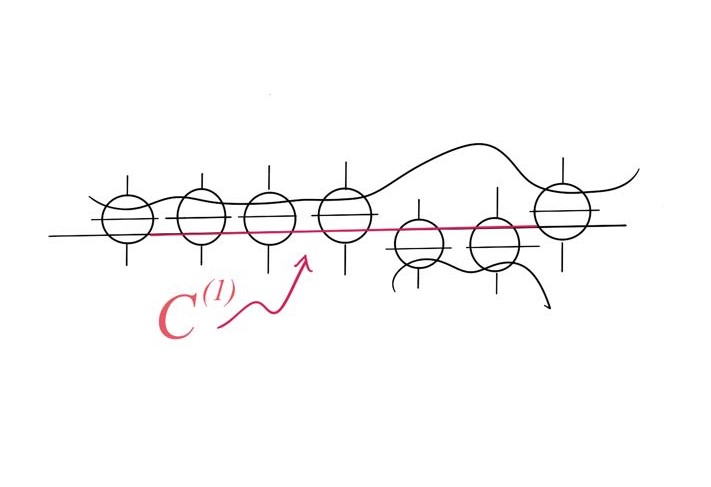}
  \caption{a R-$\omega$-reducible partial word $\omega(C^{(1)})$ reads as $RSRSRSSRS$}
  \label{F2a}
\end{subfigure}
\begin{subfigure}{.49\textwidth}
  \centering
  \includegraphics[width=.8\linewidth]{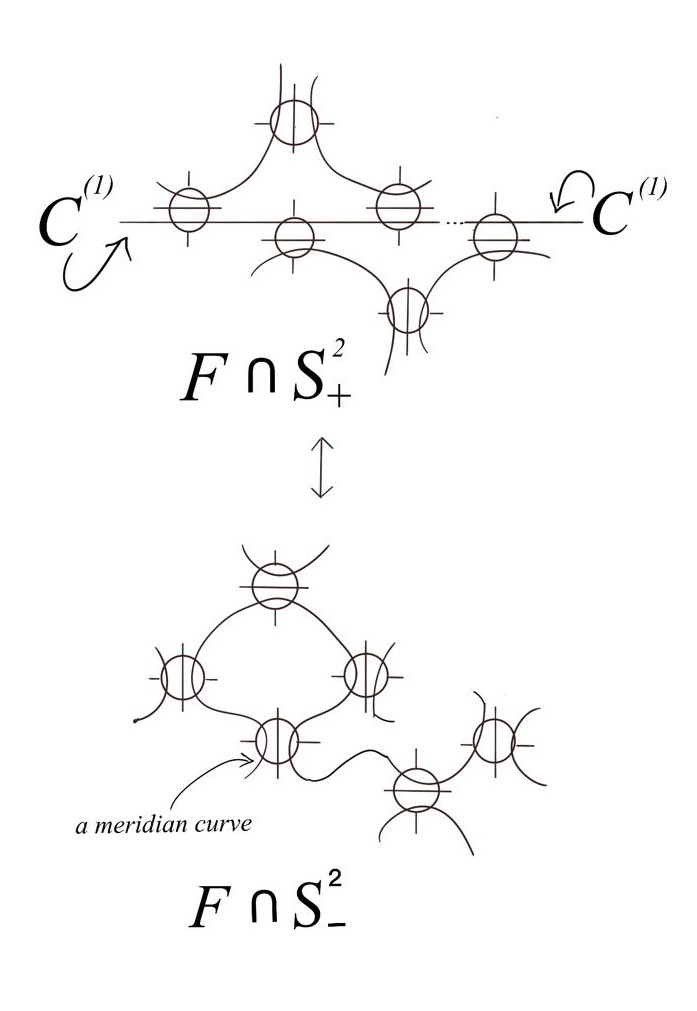}
  \caption{existence of a meridian curve}
  \label{F2b}
\end{subfigure}
\caption[Two different situations where saddles of $\omega(C^{(1)})$ are paired up with each other.]{Examples of two different situations where saddles of $\omega(C^{(1)})$ are paired up with each other.}
\label{F2}
\end{figure}

\begin{figure}[h] 
\centering
\includegraphics[width=0.6\textwidth]{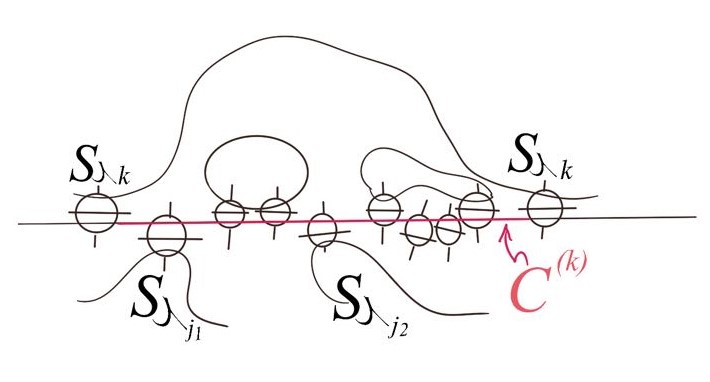}
\caption[A contradictory situation in Lemma \ref{lemma}]{A R-$\omega$-reducible partial word $\omega(C^{(k)})$ reads as $SSRSSSSRSSR$ (omitting the subscripts). $\omega(C^{(k)})$ does not contain any $S_{\lambda_{k}}$'s. $S_{\lambda_{j_1}}$, $S_{\lambda_{j_2}}$ are two saddles contained in $\omega(C^{(k)})$ that are not paired up with any saddles contained in $\omega(C^{(k)})$, and they are not recorded in the partial word between any two paired up saddles contained in $\omega(C^{(k)})$.}
\centering

\label{F3}
\end{figure}

\begin{proof}

Suppose $F \subset S^{3}-L$ is a separating sphere or a closed essential surface in normal position. Without lost of generality, we assume $C\subset F \cap S^{2}_{+}$ is a simple closed curve such that $\omega (C)$ is $\omega$-irreducible, and that each of the saddles in $\omega(C)$ is paired up to other saddle(s) in $\omega(C)$. Then each saddle can be assigned with a unique connecting operation $f_{\lambda_{i}}$, $1\leq \lambda_i \leq n$. $\omega(C)$ always consists of $S^{i}$ after we perform finitely many times of $($\RN{1}$)$, $($\RN{2}$)$, $($\RN{3}$)$ reductions since it is by assumption $\omega$-irreducible. \par

Consider an adjacent pair of saddle $S_{\lambda_{1}}$'s associated to $f_{\lambda_{1}}$ (adjacent as in one of the arcs contained in $C$ which is between this pair of saddles does not intersect with any  $S_{\lambda_{1}}$). Denote the two partial words of the arcs separated by cutting the loop $C$ at two of (the four) edge points of two type-$II$ arcs intersecting this pair of saddles as $\omega(C^{(1)})$ and $\omega(C^{(1)(1)})$, so that $\omega(C^{(1)})$ does not consist of any $S_{\lambda_{1}}$'s and $\omega(C^{(1)(1)})$ contains all the $S_{\lambda_{1}}$'s of $\omega(C)$. Note that the length of both $\omega(C^{(1)})$ and $\omega(C^{(1)(1)})$ are odd. Therefore at least one of these partial words is R-$\omega$-irreducible, otherwise the word $\omega(C)$ would be $\omega$-reducible. We can assume $\omega(C^{(1)})$ is R-$\omega$-irreducible, because otherwise we can pick a different adjacent pair of $S_{\lambda_{1}}$'s associated to $f_{\lambda_{1}}$ so that $\omega(C^{(1)})$ is R-$\omega$-irreducible and it does not consist of any $S_{\lambda_{1}}$'s. Notice if such pair of $S_{\lambda_{1}}$'s does not exists, $\omega (C)$ is again $\omega$-reducible. \par 

Not all of the saddles of $\omega(C^{1})$ are paired up with each other. Otherwise, either $\omega(C^{1})$ is R-$\omega$-reducible, since simultaneously ``deleting'' the $S$'s assigned with the same connecting operation one ``batch'' at a time will naturally form the $\omega$-reduction(s) and reduce the partial word $\omega(C^{1})$ to an odd number of $R$, see Fig. \ref{F2a}. Then we perform the $($\RN{1}$)$ reduction and we are left with only one $R$; Or $F$ is not in normal position and it contains a meridian curve, since the ``nesting behavior'' of the simple closed curves of $F\cap S^{2}_{\pm}$ will force to manifest a loop on $F\cap S^{2}_{-}$ passing through both sides of a bubble, contradicting (2) of normal position, see Fig. \ref{F2b}.\par

Now assume $\omega(C^{(k-1)})$ is $R$-$\omega$-irreducible for some $k>1$, we claim that we can find a connecting operation $f_{\lambda_{k}}$ and two paired up saddle $S_{\lambda_{k}}$'s, together with the two partial words in between two of (the four) edge points of two type-$II$ arcs, $\omega(C^{(k)})$ and $\omega(C^{(k)(k)})$, $C= C^{(k)}\cup C^{(k)(k)}$, so that:\par

($\Rn{1}$) $\omega(C^{(k-1)})$ consists of $\omega(C^{(k)})$, and $\omega(C^{(k)})$ is $R$-$\omega$-irreducible, $\omega(C^{(k)})$ does not consist of any $S_{\lambda_{k}}$'s; \par

($\Rn{2}$) Each partial word contained in $\omega(C^{(k)})$, that is also in between a pair of adjacent paired up saddles of $\omega(C^{(k)})$, is $R$-$\omega$-reducible.\par

($\Rn{3}$) $\omega(C^{(k)})$ contains saddle(s) that are not paired up with any saddles contained in $\omega(C^{k})$. \par

Because the partial word $\omega(C^{(k)})$ satisfying ($\Rn{1}$) is guaranteed to exist, otherwise we can make statements similar to the above paragraphs to show a contradiction that implies $\omega(C^{(k-1)})$ being $R$-$\omega$-reducible. We can show ($\Rn{3}$) is true if ($\Rn{1}$) is true, for similar reasons in the above arguments. We now claim that ($\Rn{2}$) is also true for some large enough integer. Because if ($\Rn{2}$) is not true for the integer $k$, we can take an adjacent paired up saddles $S_{\lambda_{k+1}}$'s contained in $\omega(C^{(k)})$, associated to some connecting operation $f_{\lambda_{k+1}}$, such that the partial word $\omega(C^{(k+1)})$ contained in $\omega(C^{(k)})$, and in between this pair of saddles, is $R$-$\omega$-irreducible. In other words, if $\omega(C^{(k)})$ does not satisfy ($\Rn{2}$), we can always find the ``inner'' partial word $\omega(C^{(k+1)})$ contained in $\omega(C^{(k)})$ that is $R$-$\omega$-irreducible, and does not consist of any $S_{\lambda_{k+1}}$'s. Again $\omega(C^{(k+1)})$ satisfies ($\Rn{1}$), ($\Rn{3}$), with the subscripts in the two conditions modified. Eventually, by finiteness of the word length, we can always find an ``innermost'' partial word satisfying condition ($\Rn{1}$), ($\Rn{2}$) and ($\Rn{3}$) (together with a ``second innermost'' partial word, whose subscript is 1 smaller than the ``innermost'' one). Note that we just need to modify the subscripts to make the final claim.\par

Let $\{S_{\lambda_j}\}$ be all the saddle(s) contained in $\omega(C^{k})$ that are not paired up with any saddles contained in $\omega(C^{k})$, and each saddle of $\{S_{\lambda_j}\}$ is not recorded in the partial word between any two paired up saddles contained in $\omega(C^{k})$, see Fig. \ref{F3}. Notice such saddle(s) exist, otherwise we have a contradiction by previous arguments.\par

Suppose the cardinality of $\{S_{\lambda_j}\}$ is larger or equal to two. According to ($\Rn{2}$), we can first perform $($\RN{2}$)$, $($\RN{3}$)$ reductions on $S$'s which correspond to saddles paired up within $\omega(C^{k})$ and the $R$-$\omega$-reducible partial words in between, then perform $($\RN{2}$)$, $($\RN{3}$)$ reduction(s) on the saddles of $\{S_{\lambda_j}\}$ and we are left with an odd number of $R$'s, see Fig. \ref{F3}, then we perform the $($\RN{1}$)$ reduction and we are left with only one $R$, contradicting $\omega(C^{k})$ being $R$-$\omega$-irreducible; \par

Suppose the cardinality of $\{S_{\lambda_j}\}$ is one, the resulting curves on $F \cap S^{2}_{-}$ are similar to Fig. \ref{F2b}, contradicting (2) of normal position. 

For a surface with meridional boundaries (punctures) that is placed in normal position, the idea and the technical definitions needed are almost identical to the above proof.\\

\end{proof}

\subsection{Virtual Word and Virtual Diagram}

The rest of definitions in this section are given in order to convert the word problem of multiple saddles contained in a single bubble to the word problem on which every bubble a selected loop $C$ intersects contains only one saddle, so that Lemma \ref{lemma} would apply. The purpose is to obtain a bookkeeping diagram, in which we obtain the \emph{virtual word} $\omega^v(C)$ by recording labels of the \emph{virtual bubbles} $C$ intersects and the labels on type $I$ arcs of $C$. These definitions are unrelated to the concept of virtual crossing. \par

Assume $F\subset S^{3}-L$ is a connected surface that satisfies all the conditions of normal position, with the possible exception of (2).\par

\theoremstyle{definition}
\begin{definition}[virtual bubble] \label{bubble}

Let $C\subset F\cap S^{2}_{+}$ be a simple closed curve. Then from $B^{3}_{+}$'s side of view, the \emph{virtual bubbles} are a bunch of bubbles intersecting $C$ obtained by replacing a bubble $C$ intersects that contains $1+l$ saddles with a set of $1+2l$ bubbles. Here $l \geq 0$, $C$ intersects the saddle $\overset{_l}{S}$ with $l$ saddles underneath it, and $\overset{_l}{S}$ is not necessarily the top most saddle in the bubble. These virtual bubbles intersect $C$ alternatingly in the following way: see Fig. \ref{F4}. \par
\begin{figure}[h]
\begin{subfigure}{.49\textwidth}
  \centering
  \includegraphics[width=.8\linewidth]{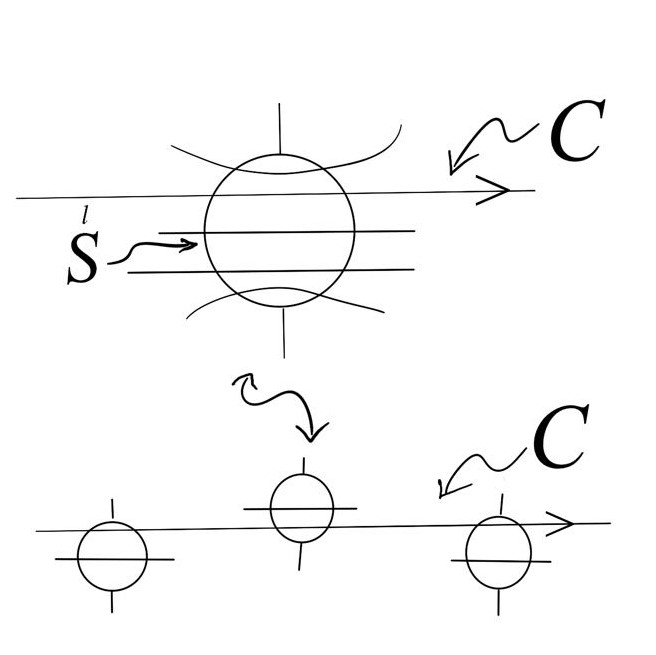}
  \caption{when $l$ is odd}
  \label{F4a}
\end{subfigure}
\begin{subfigure}{.49\textwidth}
  \centering
  \includegraphics[width=.8\linewidth]{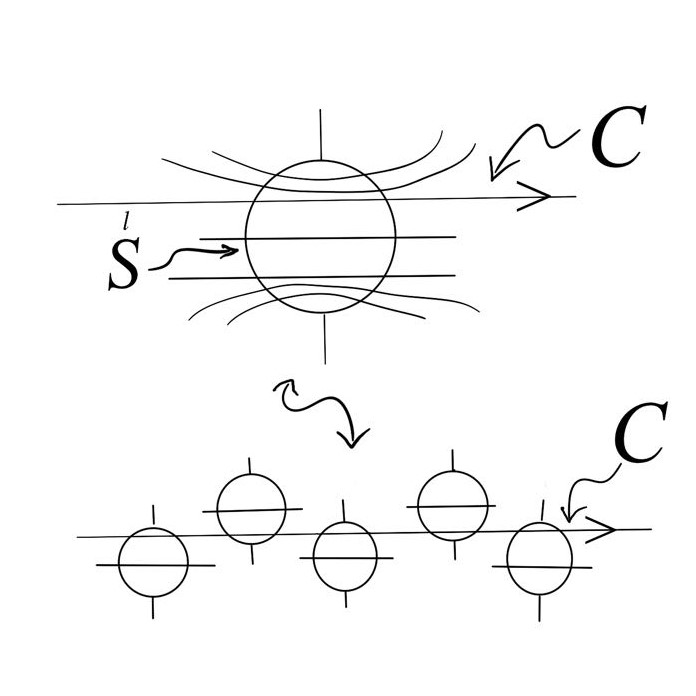}
  \caption{when $l$ is even}
  \label{F4b}
\end{subfigure}
\caption[Virtual bubbles.]{Virtual bubbles. As we perform this bubble replacement, $C$ is oriented in the same direction, and the ``outermost'' virtual bubble manifested is on the same side of $C$ as the original bubble.}
\label{F4}

\end{figure}

\end{definition}

\theoremstyle{definition}
\begin{definition}[virtual word] \label{virtual word}
With all bubbles that $C$ intersects replaced by virtual bubbles, whose corresponding type $II$ arcs are each labeled with an $S$, we obtain a \emph{virtual word} $\omega^{v}(C)$, that is a record in order the labels of $C$, in the same sense as the record of labels in a cyclic word. Notably, we also create $2l$ new type $I$ arcs, and labels of them are $\emptyset$ when we replace a bubble that contains $1+l$ saddles with virtual bubbles, while the labels of the original type $I$ arcs remain unchanged.\par

\end{definition}

In other words, we can obtain virtual words from cyclic words in the following way that is given in the definition:\par

\theoremstyle{definition}
\begin{definition}[$l$-reduction] Assume $C$ is a curve of $F\cap S^{2}_{+}$ (or $F\cap S^{2}_{-}$), $\overset{_l}{S}$ intersects $C$, and from $B^{3}_{+}$'s side of view (or from $B^{3}_{-}$'s side of view, respectively), there are totally $l$ saddles underneath it. We can replace the associated $S$ of $\overset{_l}{S}$ in the cyclic word $\omega(C)$ with the notation $\overset{_l}{S}$. We call the following operation performed on a cyclic word $\omega(C)$ an \emph{$l$-reduction}:\par
$($\RN{5}$)$ $\overset{_l}{S}\to S^{2l+1}$.\par
$l$ is a non-negative integer. The resulting word is a virtual word $\omega^{v}(C)$. \par

\end{definition}

Similar to the cyclic word, we can define $\omega$-reducible or $\omega$-irreducible virtual words.\par

\theoremstyle{definition}
\begin{definition}[$\omega$-reducible virtual word]
We say a virtual word $\omega^{v}(C)$ is \emph{$\omega$-reducible} if it is $\emptyset$ after we perform finitely many times of $($\RN{1}$)$, $($\RN{2}$)$, $($\RN{3}$)$, $($\RN{4}$)$ reductions, otherwise we call it \emph{$\omega$-irreducible}.\par
\end{definition}

\begin{figure}[h]
\begin{subfigure}{.49\textwidth}
  \centering
  \includegraphics[width=.75\linewidth]{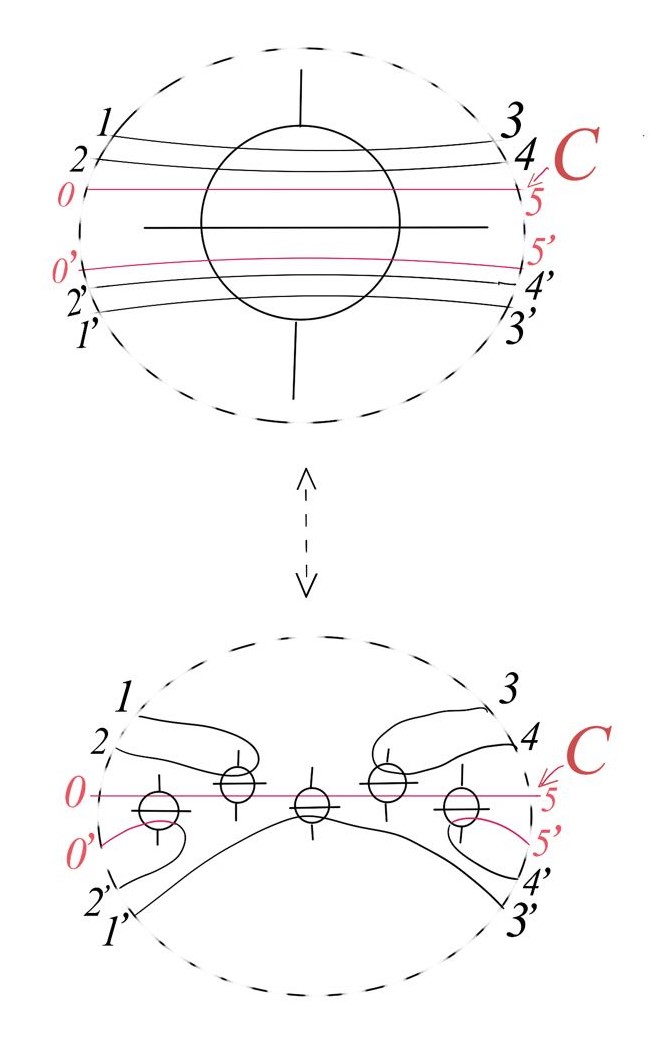}
  \caption{We replace a $(1+l)$-saddle-bubble with $2l+1$ virtual bubbles arranged alternatingly. The dotted ellipse record the related arc-ends in order.}
  \label{vb}
\end{subfigure}
\begin{subfigure}{.49\textwidth}
  \centering
  \includegraphics[width=1.1\linewidth]{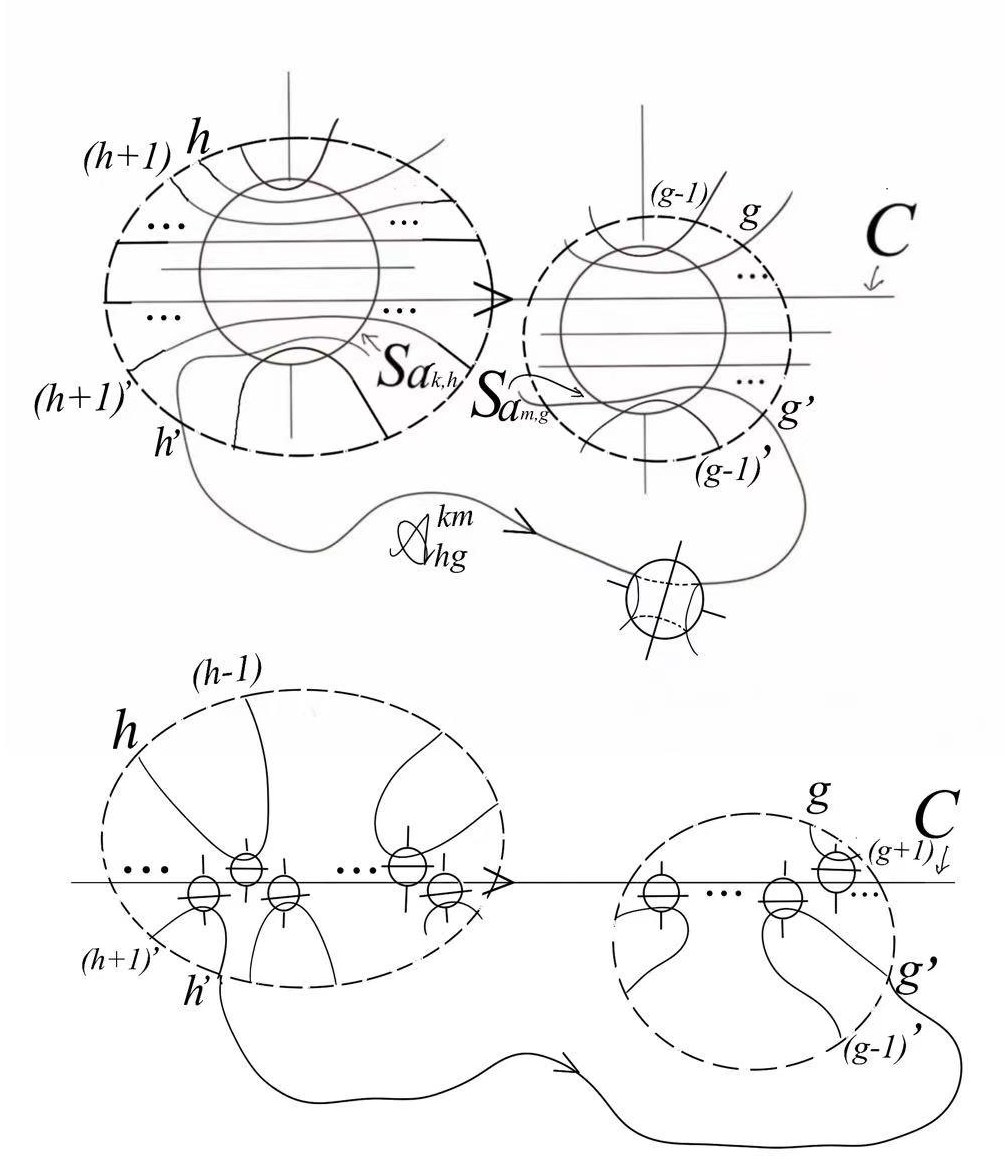}
  \caption{The arc $\mathscr{A}^{km}_{hg} \subset F\cap S^{2}_{-}$ intersects both $S_{\alpha_{k,h}}$ and  $S_{\alpha_{m,g}}$, therefore the corresponding virtual bubbles are connected. In particular we connect the arc-end marked $h'$ to the arc-end marked $g'$.}
  \label{F5}
\end{subfigure}
\caption[Virtual diagrams]{}
\label{conpound}

\end{figure}

The following definition of \emph{virtual diagram} is the construction of a bookkeeping diagram, so that it and its dual diagram would satisfy Proposition \ref{theorem2}.\par

\theoremstyle{definition}
\begin{definition}[virtual diagram]\label{virtual}
Suppose $F\subset S^{3}$ is a connected surface satisfying (1), (3), (4), (5) and (6) of normal position, $C\subset F \cap S^{2}_{+}$. A \emph{positive virtual diagram of $C$}, denoted as $F^{v+}_{C}\cap S^{2}_{+}$ (or respectively, if $C'\subset F \cap S^{2}_{-}$, a \emph{negative virtual diagram of $C'$}, $F^{v-}_{C'}\cap S^{2}_{-}$), is a diagram devoid of link and surface structure, obtained by modifying (part of) the diagram $F\cap S^{2}_{+}$ in the following way:\par

Suppose $S_{\alpha_{k,h}}$ and  $S_{\alpha_{m,g}}$ both intersect an arc $\mathscr{A}^{km}_{hg} \subset F\cap S^{2}_{-}$, then we connect the virtual bubbles manifested by the bubbles containing $S_{\alpha_{k,h}}$ and  $S_{\alpha_{m,g}}$ as shown in Fig. \ref{conpound}. We perform such operations on each virtual bubble of $C$ to obtain $F^{v+}_{C}\cap S^{2}_{+}$. \par

The dual diagram of a positive virtual diagram $F^{v+}_{C}\cap S^{2}_{+}$, is denoted as $F^{v+}_{C}\cap S^{2}_{-}$, produced in the same way that is shown in Fig. \ref{F0}.\par

\end{definition}

\begin{figure}[h]
\begin{subfigure}{.49\textwidth}
  \centering
  \includegraphics[width=.8\linewidth]{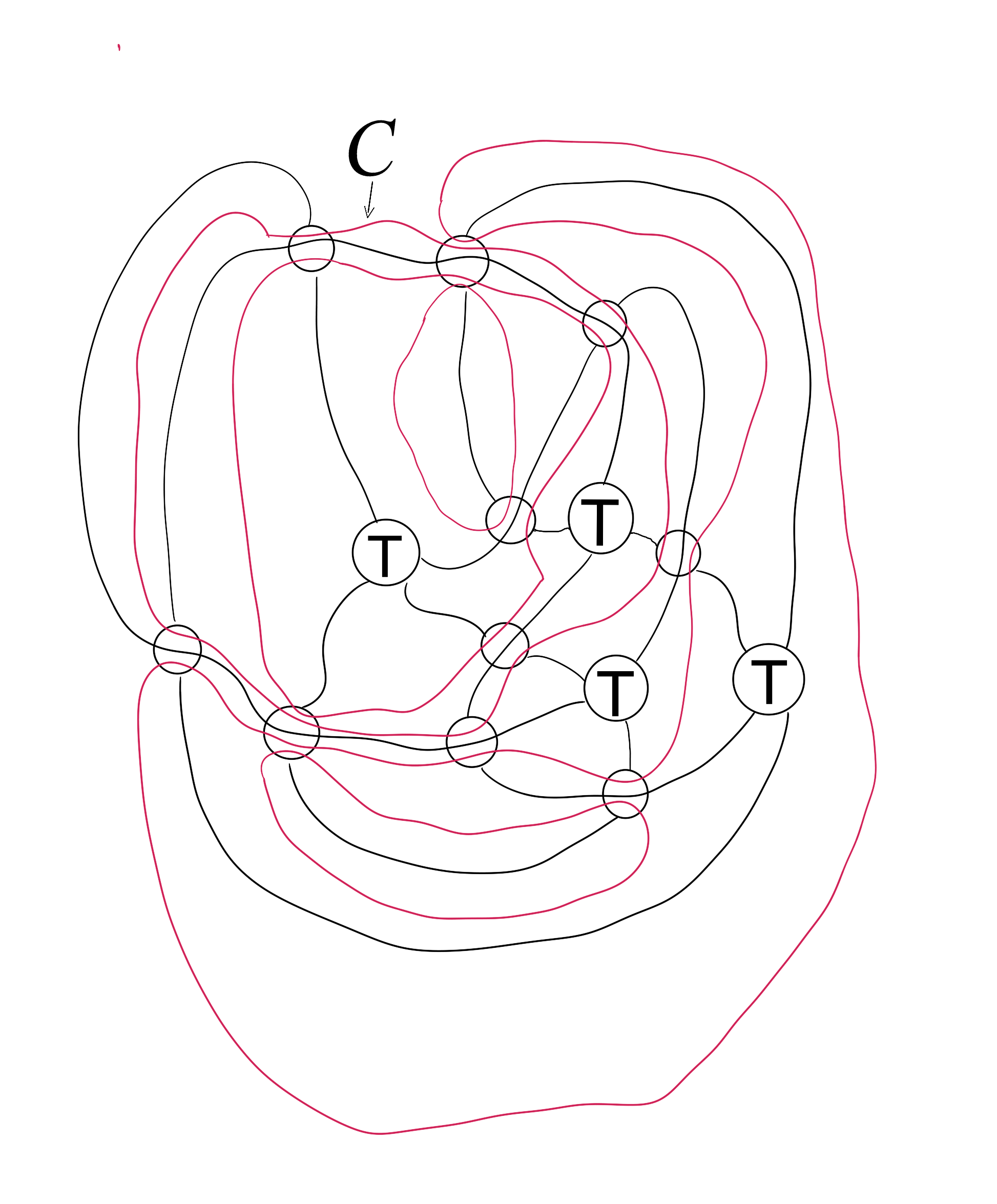}
  \caption{$F \cap S^{2}_{+}$}
  \label{F6a}
\end{subfigure}
\begin{subfigure}{.49\textwidth}
  \centering
  \includegraphics[width=.8\linewidth]{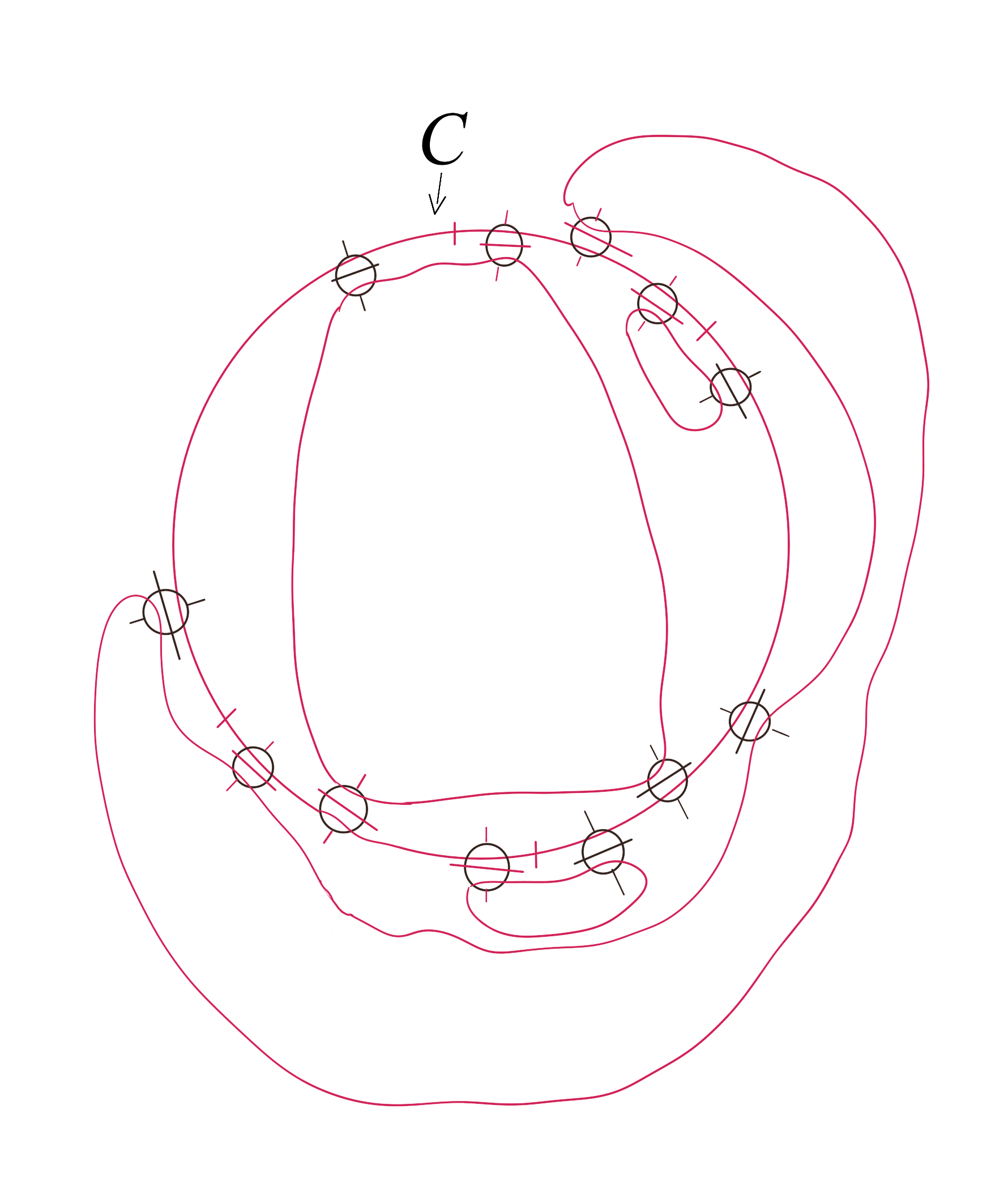}
  \caption{$F^{v+}_{C}\cap S^{2}_{+}$}
  \label{F6b}
\end{subfigure}
\caption[A diagram and a positive virtual diagram of a torus in link complement.]{(a) shows a torus $F$ (red colored) embedded in a link (black colored) complement from $B^{3}_{+}$'s side of view. (b) is the positive virtual diagram of $C$, in which the 4 red slashes on $C$ separate 2 arcs each intersecting 3 virtual bubbles, which are manifested from the the bubble containing multiple saddles. These virtual bubbles are marked with red ``link strands''.}
\label{F6}

\end{figure}

The following proposition transfers the word problem of multiple-saddle case to the case where each bubble a selected loop $C$ crosses contains only one saddle, so that Lemma \ref{lemma} would apply for the corresponding virtual word:\par

\begin{proposition} \label{theorem2}
Suppose $F\subset S^{3}$ is a connected surface satisfying condition (1), (3), (4), (5), (6) of normal position, $C \subset F \cap S^{2}_{+}$. Then if in $F^{v+}_{C}\cap S^{2}_{-}$, there exists a loop going through a virtual bubble manifested from a bubble $B$ twice, or two virtual saddles manifested from $B$ are paired up, then there exists a loop in the diagram of $F\cap S^{2}_{-}$ that goes through $B$ twice. \par

\end{proposition}

\begin{proof}

This is by construction of the virtual diagrams, see Definition  \ref{bubble}, Definition  \ref{virtual}. See also Example \ref{torusexample}.\par

\end{proof}

\subsection{An Example of Diagram, Virtual Diagram and Dual Diagrams}

\begin{example}\label{torusexample}

 Fig. \ref{F6a} shows a torus $F$ embedded in a link complement, and it is placed in normal position. We are interested in a particular s.c.c., $C \subset F \cap S^{2}_{+}$. From $F \cap S^{2}_{+}$, we can generate the positive virtual diagram of $C$, denoted as $F^{v+}_{C}\cap S^{2}_{+}$, see Fig. \ref{F6b}. The dual diagrams of $F \cap S^{2}_{+}$ and $F^{v+}_{C}\cap S^{2}_{+}$ are pictured in Fig.\ref{F7}.\par
\end{example}

\begin{figure}[h]
\begin{subfigure}{.49\textwidth}
  \centering
  \includegraphics[width=.8\linewidth]{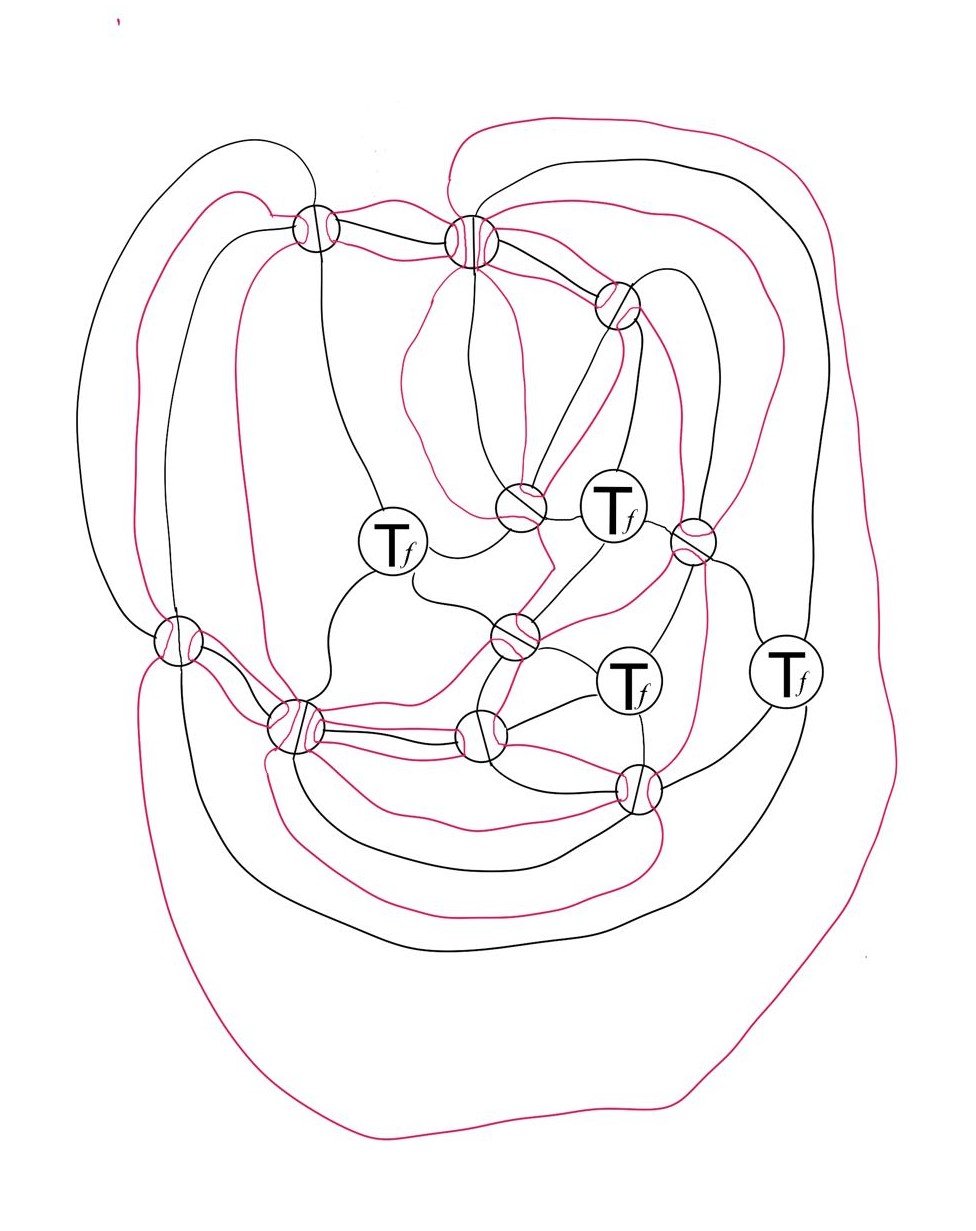}
  \caption{$F \cap S^{2}_{-}$}
  \label{F7a}
\end{subfigure}
\begin{subfigure}{.49\textwidth}
  \centering
  \includegraphics[width=.8\linewidth]{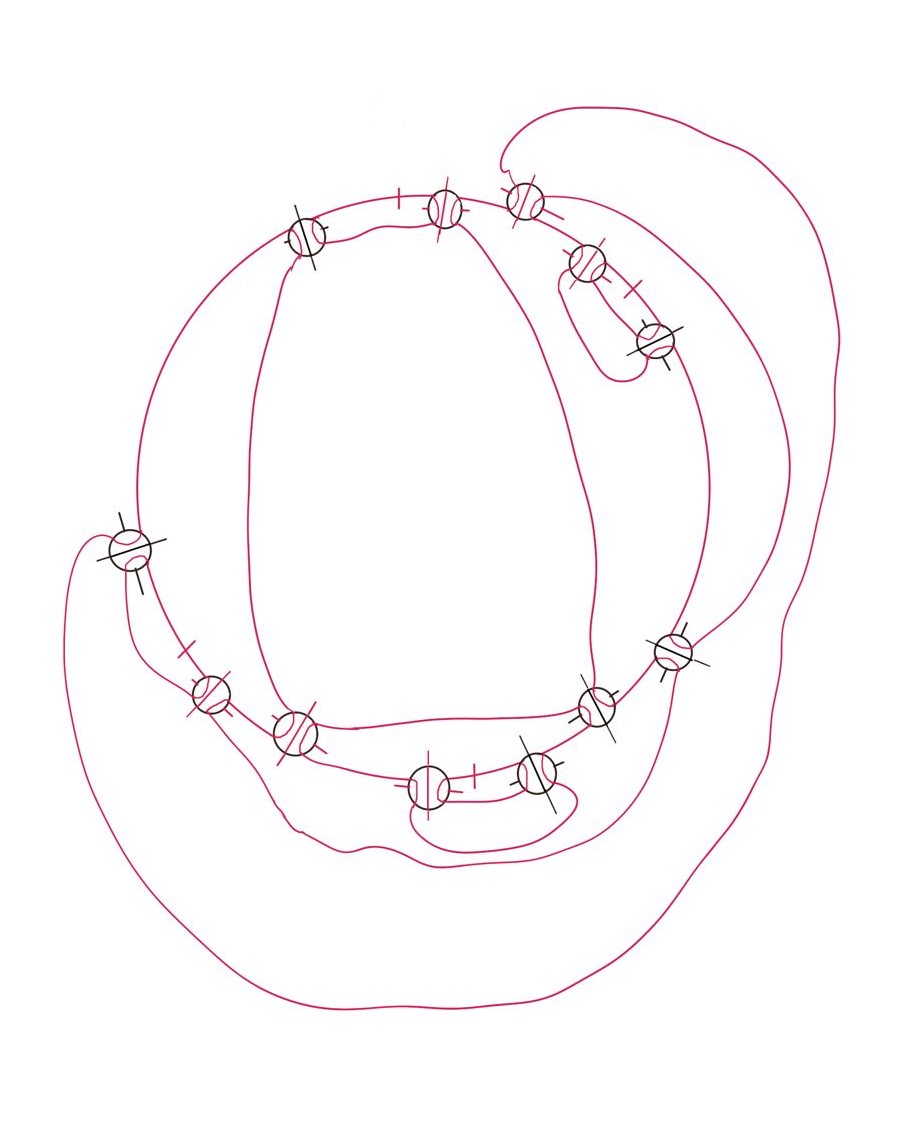}
  \caption{$F^{v+}_{C}\cap S^{2}_{-}$}
  \label{F7b}
\end{subfigure}
\caption{The dual diagrams of Figure \ref{F6}}
\label{F7}

\end{figure}

\subsection{Proof of Theorem \ref{Main}}

\begin{proof}
According to Proposition \ref{sp1} and Proposition \ref{sp2}, we can assume without loss of generality $F$ is a separating sphere or an essential surface in normal position. And suppose $\omega^{v}(C)$ is $\omega$-irreducible for some $C\subset F \cap S^{2}_{+}$, then there exists a loop in the dual virtual diagram $F^{v-}_{C}\cap S^{2}_{+}$ passing through a virtual bubble twice, according to Lemma \ref{lemma}. Then according to Proposition \ref{theorem2}, either there is a loop in $F\cap S^{2}_{-}$ passing through the same side of a bubble twice, contradicting the minimality of saddle number; Or there is a loop in $F\cap S^{2}_{-}$ passing through different sides of a bubble twice, contradicting the pairwise incompressibility of $F$.\par
\end{proof}

\section{An Upper Bound On The Euler Characteristic}\label{chapter final}

\subsection{Pullback Graph and A Complexity Measurement}

\begin{figure}[h]
\begin{subfigure}{.49\textwidth}
  \centering
  \includegraphics[width=.95\linewidth]{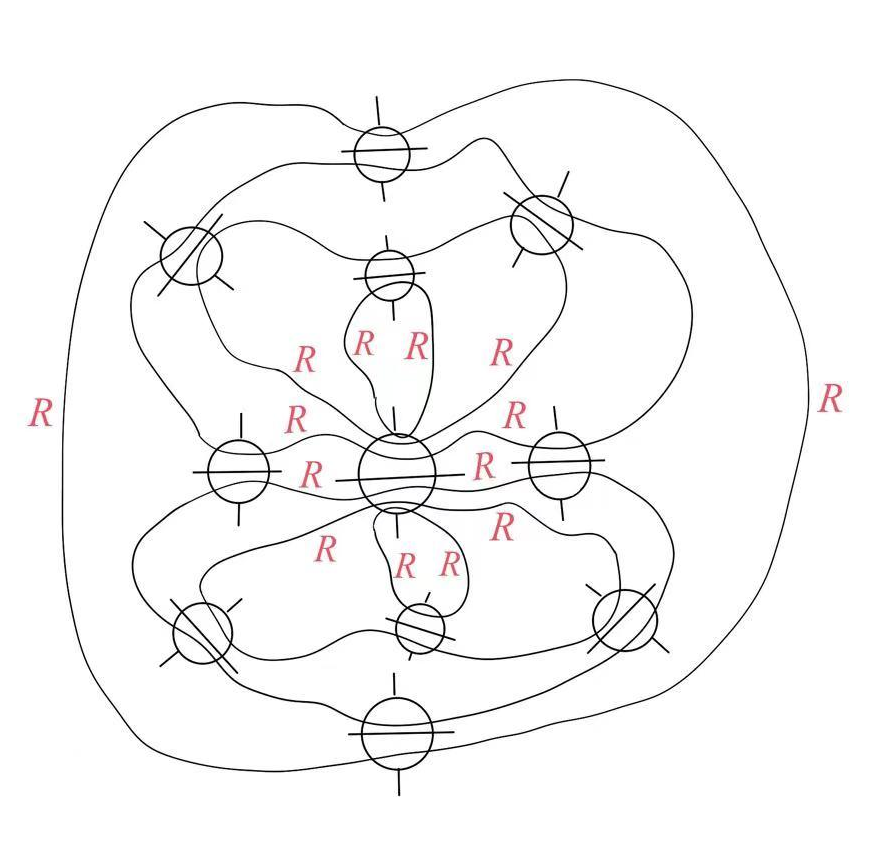}
  \caption{$F \cap S^{2}_{+}$}
  \label{flake}
\end{subfigure}
\begin{subfigure}{.49\textwidth}
  \centering
  \includegraphics[width=.95\linewidth]{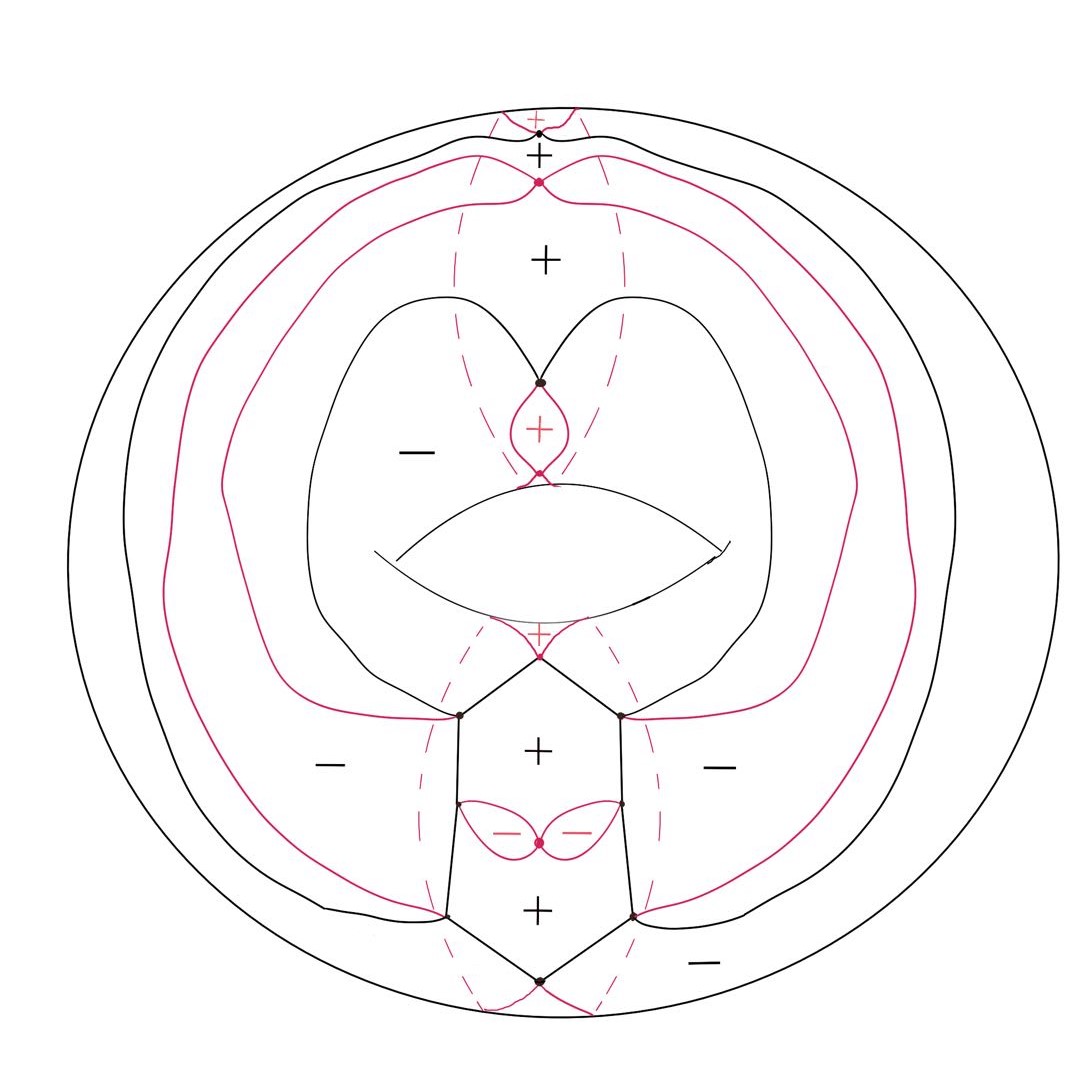}
  \caption{The pullback graph of $F$.}
  \label{pullback}
\end{subfigure}
\caption[The pullback diagram of surfaces in normal position.]{The edges marked in $R$ are colored in red in the pullback graph of $F$. Disks above $S^{2}_{+}$ are marked with $+$, while disks below above $S^{2}_{-}$ are marked with $-$.}
\label{pullback graph}

\end{figure}

Let $F$ be a separating sphere, or a closed essential surface in $S^{3}-L$ that is in normal position. Consider $F$ with all intersection curves both in $F\cap S_{+}^{2}$ and $F\cap S_{-}^{2}$ projected on it. All saddles correspond to quadrilaterals, which we collapse to vertices to obtain a 4-regular graph on $F$. $F$ is therefore separated into disk regions, and it can be checkerboarded so that for any two adjacent faces in this graph, one corresponds to a disk region which is contained strictly above $S^{2}_{+}$, and the other to a disk region contained strictly below $S^{2}_{-}$, see Fig.\ref{pullback graph}. A disk region which has $n$ vertices ($n$ must be even) will be denoted as $D_{n}$. An edge can be marked with $R$ if it corresponds to a type $I$ arc labeled with $R$ in $F\cap S_{+}^{2}$, or equivalently,  $F\cap S_{-}^{2}$. We call this graph a \emph{pullback graph} of $F$. Denote $|R|$ as the total number of $R$'s in the graph, $|S|$ as the vertex number (also the number of saddles of $F$). \par

\begin{lemma}\label{algorithm}
Let $L \subset S^{3}$ be a link, $F \subset S^{3}-L$ be a separating sphere or an essential surface in normal position, then for each simple closed curve $C \subset F \cap S^{2}_{\pm}$ that intersects with saddles, $\omega(C)$ contains at least two $R$'s.  \par

\end{lemma}

\begin{proof}
Assume $\omega(C)$ contains only $P$'s and $S$'s, then according to the construction of virtual word (see Definition  \ref{bubble}), any $\omega^{v}(C)$ consists of only $P$'s and $S$'s, i.e. $C$ intersects virtual saddles in an alternating pattern, which makes any of its virtual word $\omega$-irreducible, contradicting Theorem \ref{Main}.

\end{proof}

This means in the link complement, if we put the essential surface $F$ in normal position, then in the diagrams of $F \cap S^{2}_{\pm}$ the arrangement of saddles intersecting each simple closed curve can never be alternating. The following lemma is also true if we replace $S_{+}^{2}$  with $S_{-}^{2}$.\par 

\begin{lemma}\label{lower bound}
Let $D_{n}$ be a disk component of $F\cap B_{+}^{3}$, whose boundary $\partial D_n$ is a simple closed curve intersecting $n$ saddles, and $\partial D_n$ divides $S_{+}^{2}$ into two regions. Then from $B_{+}^{3}$'s side of view, the $n$ saddles may be included in different regions. The region including $m_{0}$ saddles ($m_{0}\leq n$) that intersect $\partial D_n$ contains loop(s) of $F\cap S_{+}^{2}$, which in total admit at least $m_{0}$ type $I$ arcs that are labeled with $R$.

\end{lemma}

\begin{proof}
We assume by convention the region that includes $m_{0}$ saddles is the inner one. Suppose $m_{0}=1$, the region still contains at least one loop, admitting 2 $R$'s on the corresponding edges in pullback graph of $F$ by Lemma \ref{algorithm}. Suppose $m_{0}\geq 2$, and if the $m_{0}$ saddles are all paired up with each other, then the loop(s) contained in the region admits at least $m_{0}$ $R$'s. If $m_1$ of the saddles are not paired up, then the $m_{0}-m_{1}$ saddles are paired up through loop(s) admitting at least $m_{0}-m_{1}$ $R$'s, while the other saddles are going to induce loop(s) intersecting the not paired up saddles. Assume these loop(s) induce $k_{1}$ $R$'s and include at least $m_2$ saddles on the inner region(s) which does not include those $m_1$ saddles mentioned above, where $k_{1}$ is a positive integer, then $k_{1}+m_2\geq m_{1}$. If $m_{2} \geq 1$, we can repeat this counting process on the inner region(s) bounded by the loop(s). In the $(i+1)$th step we obtain $m_{i}-m_{i+1}$ $R$'s through paired up saddles, $k_{i+1}$ $R$'s through the not paired up saddles, and $m_{i+2}$ saddles in the inner region(s) obtained through the $(i+1)$th step, where $k_{i+1}+m_{i+2} \geq m_{i+1}$. As we repeat this process, after finitely many steps we will obtain second innermost region(s) that includes $m_s$ saddle(s) in total, and innermost loop(s) admitting at least $m_s$ $R$'s. In total, we can obtain more or equal to $m_{0}$ type $I$ arcs which are labeled with $R$.\par

\end{proof}

\begin{figure}[htp] 

  \centering

  \begin{tabular}{cc}


    \includegraphics[width=50mm]{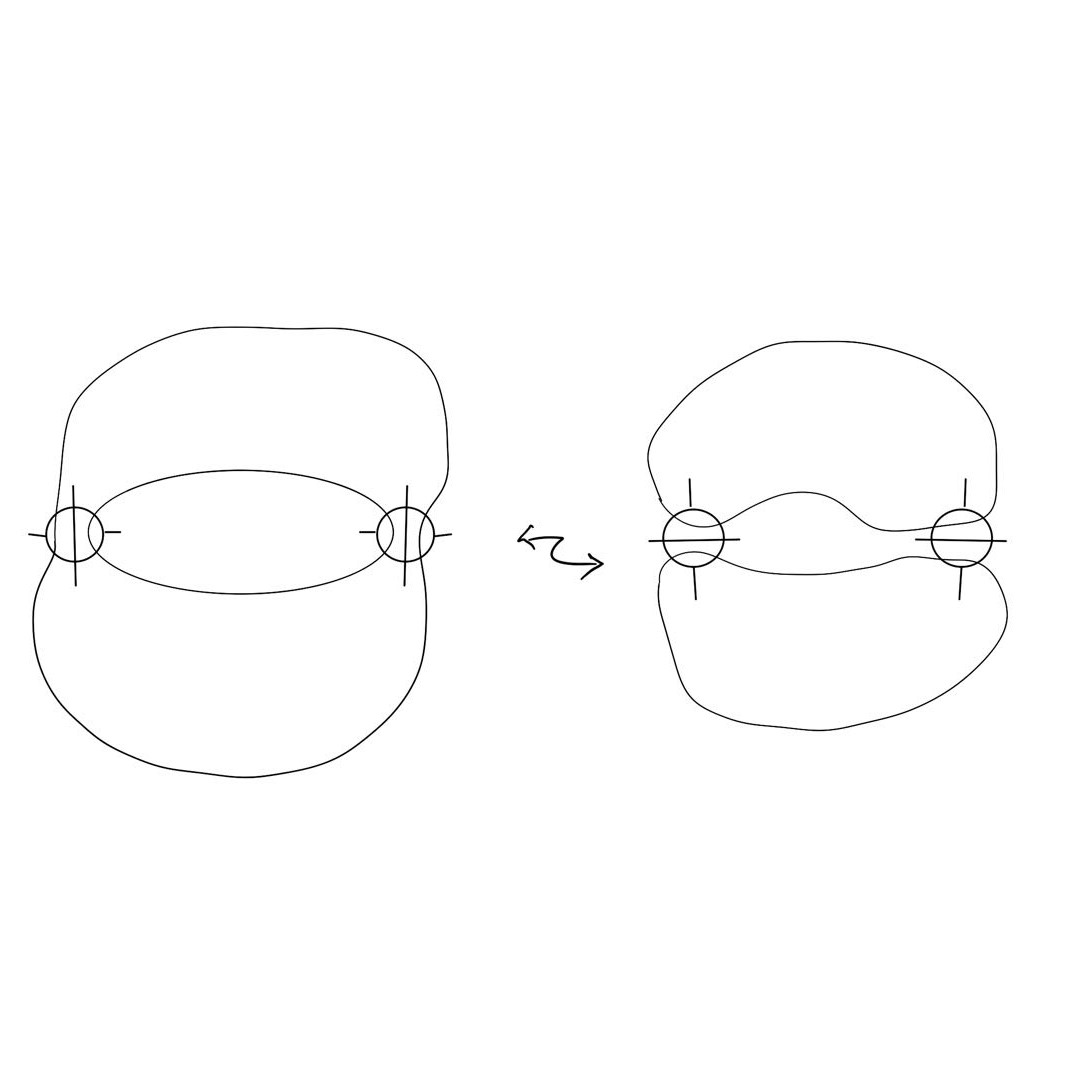}&

    \includegraphics[width=50mm]{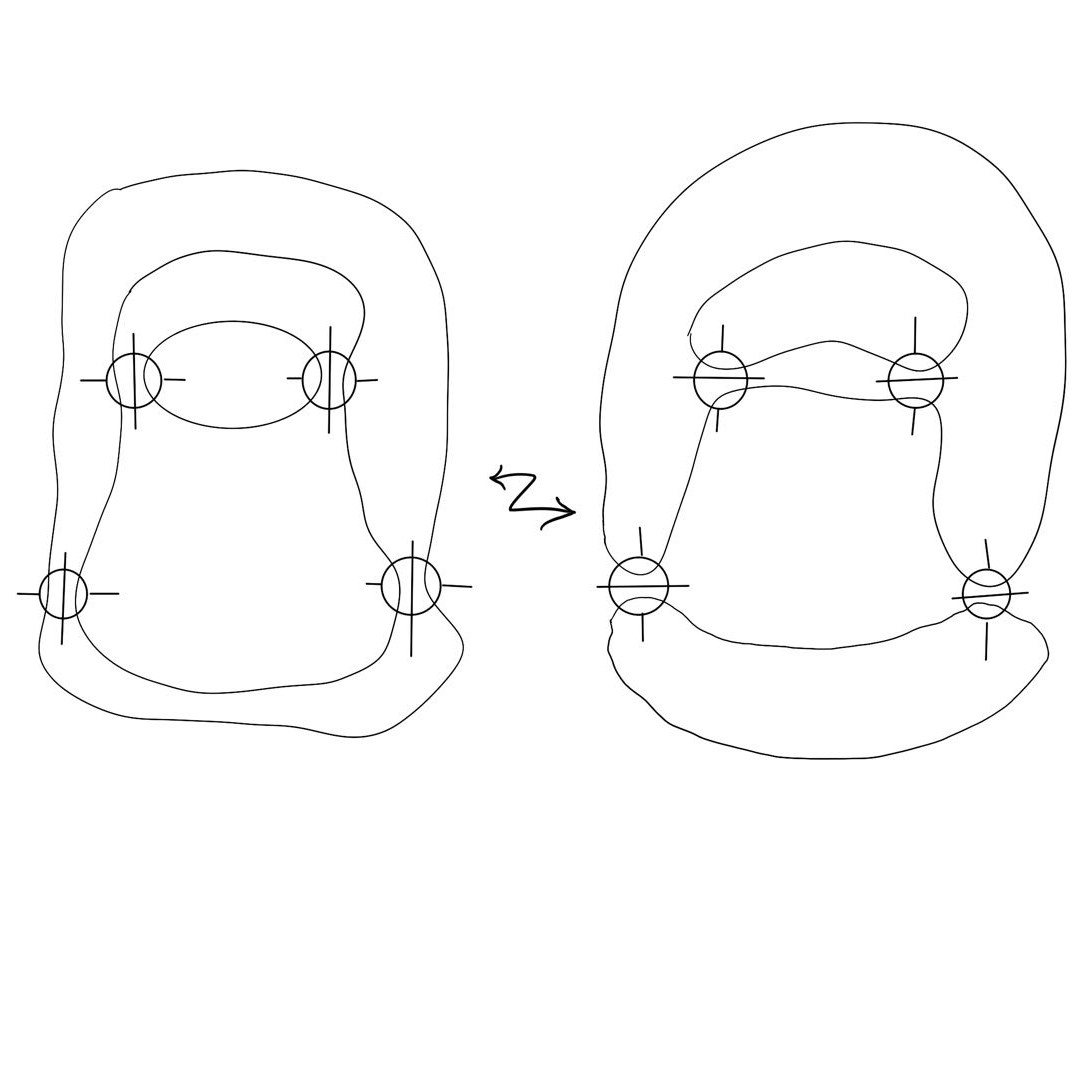}\\

    \includegraphics[width=50mm]{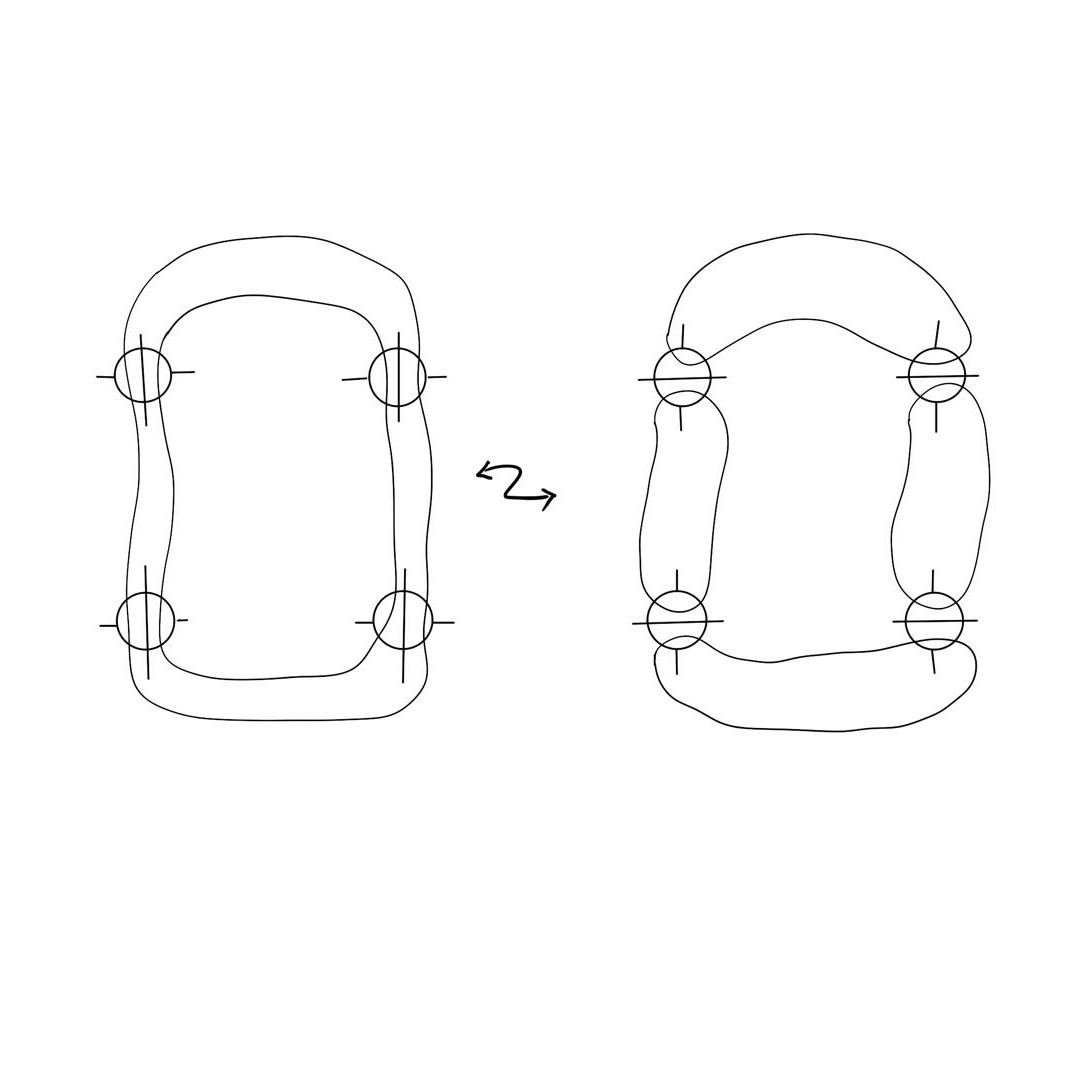}&

    \includegraphics[width=50mm]{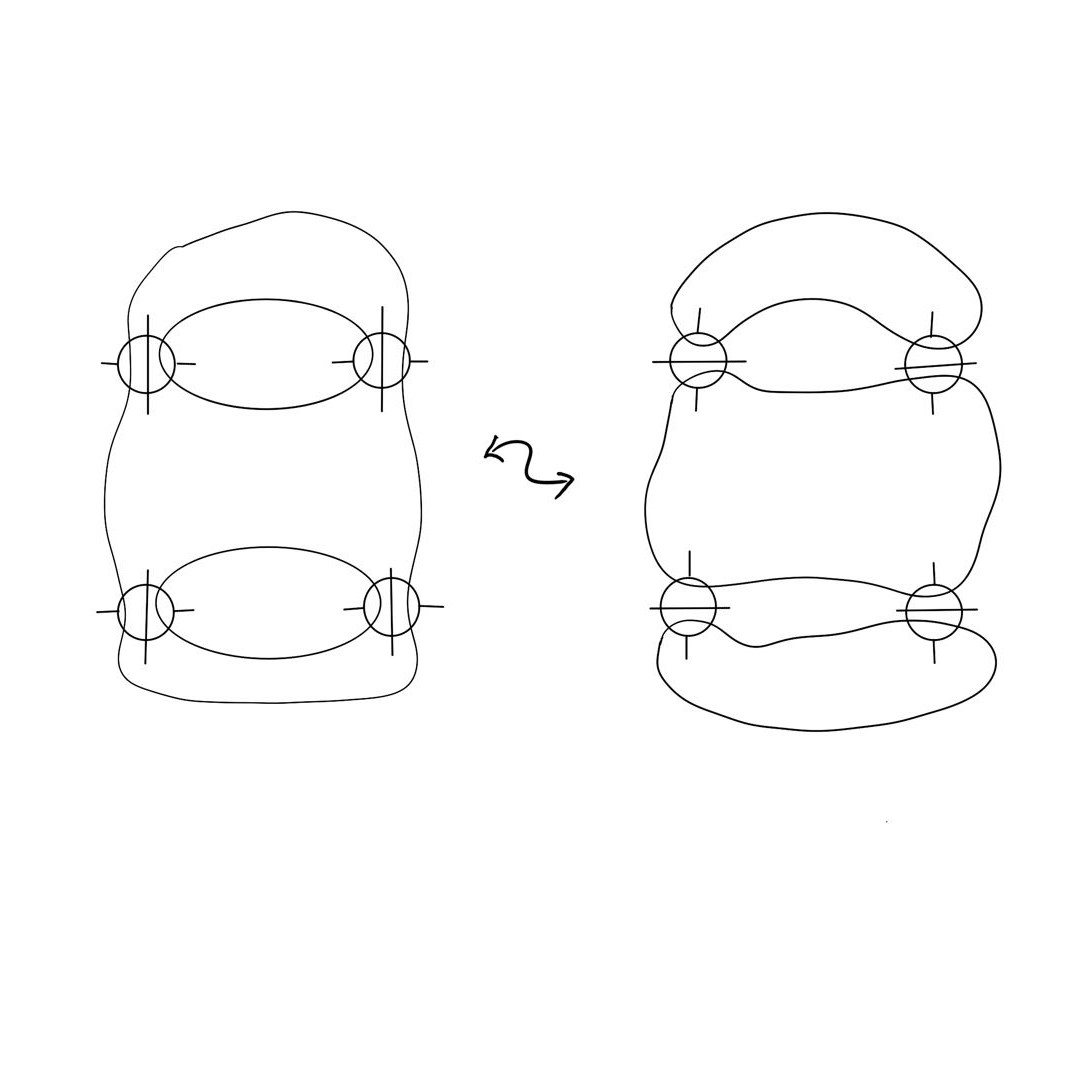}\\

    \includegraphics[width=50mm]{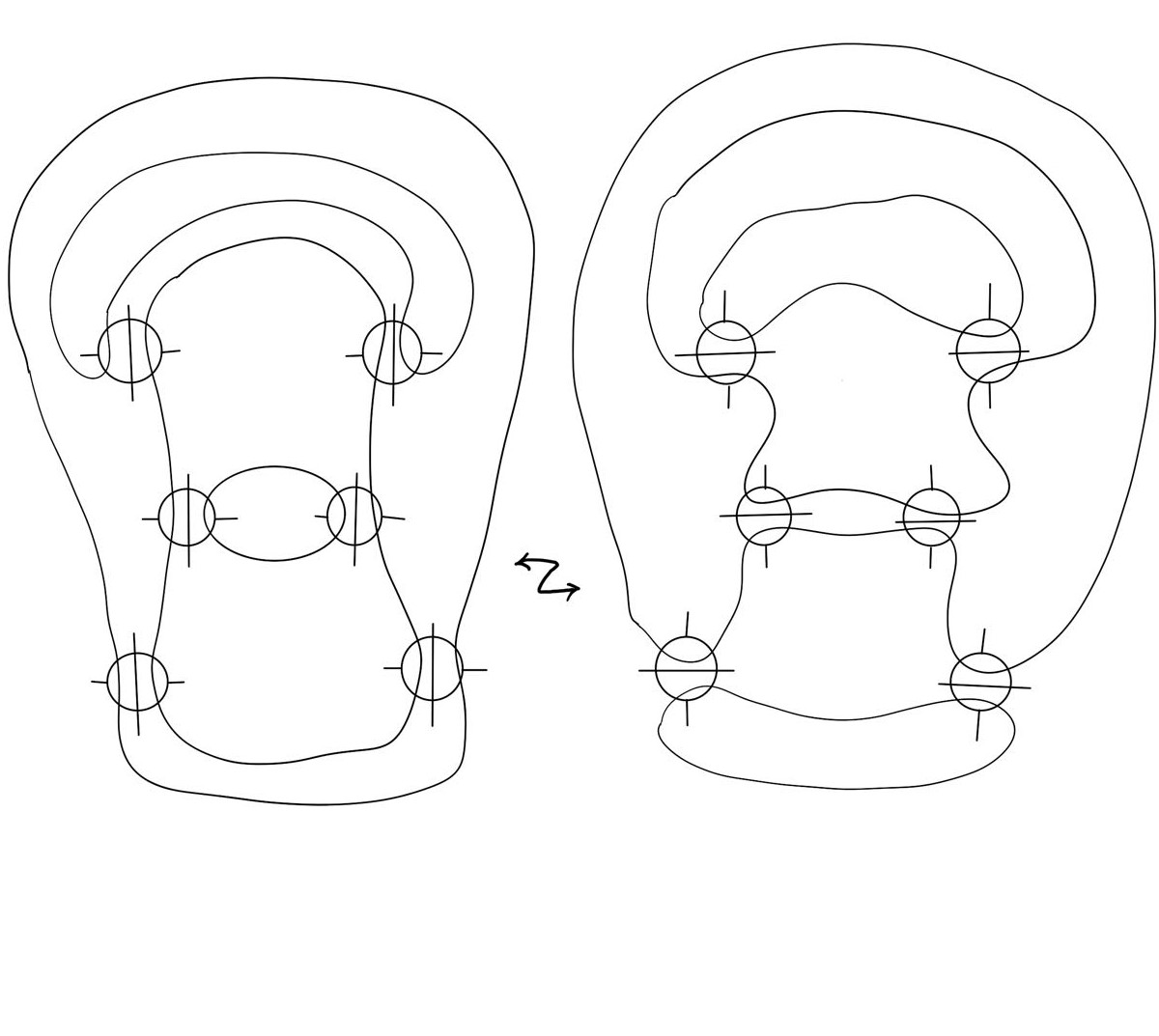}&

    \includegraphics[width=50mm]{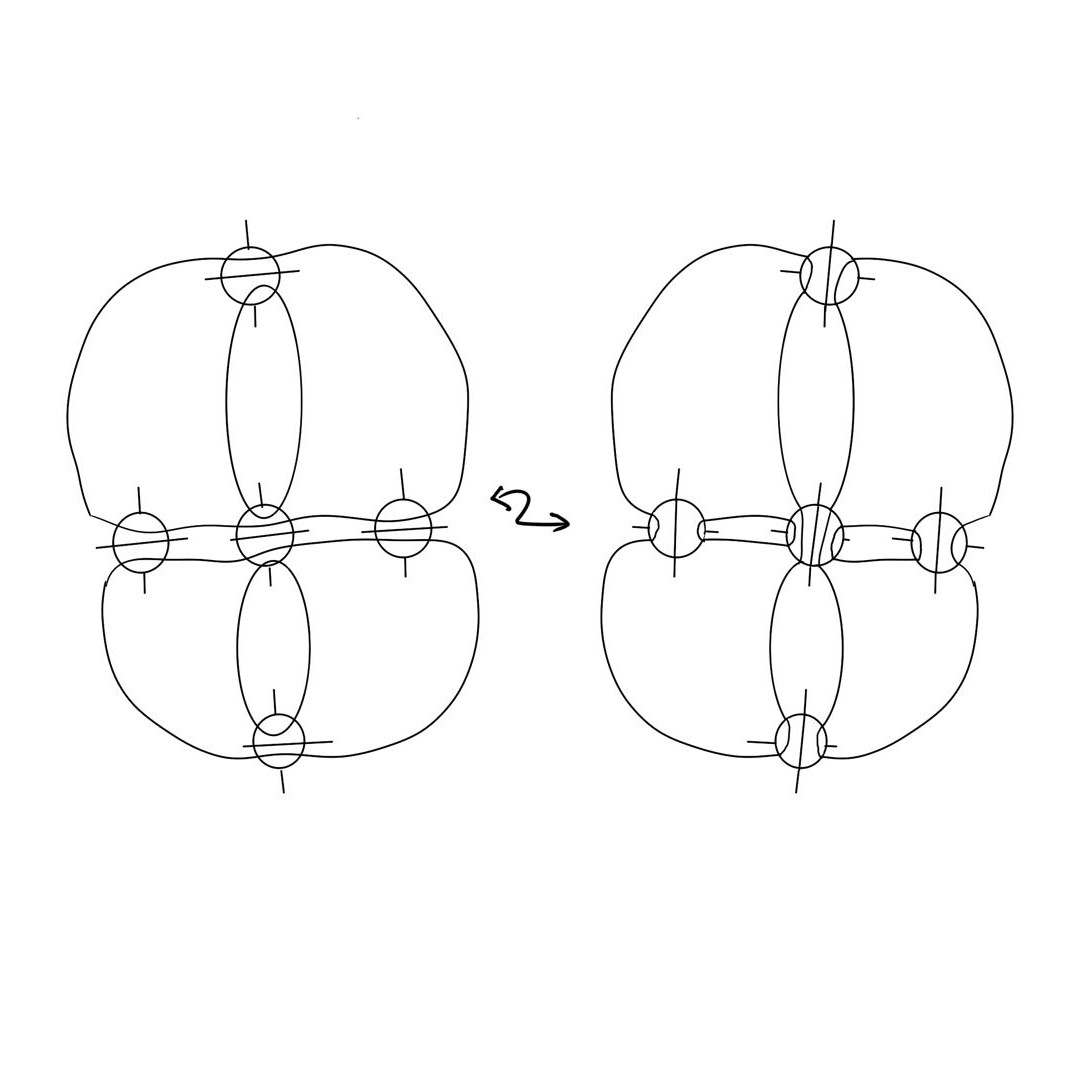}\\

  \end{tabular}
   
   \caption{The configurations of sphere when $|R|=4$, $6$, or $8$.}
\label{Rspheres}

\end{figure}

\subsection{Proof of Theorem \ref{characterization}}

\begin{proof}

To view (a), according to Lemma \ref{algorithm}, when $|R|=4$ or $6$, we only need to check the situation in which $F\cap S_{+}^{2}$ contains two or three loops. When $|R|=8$ the maximal loop number becomes 4. In each of the cases we can calculate the Euler characteristics of $F$ to show it is a sphere, see Fig.\ref{Rspheres},  except for when $|R|=8$ there are only two configurations of torus shown in Fig. \ref{R8torus}. \par

\begin{figure}[htp]

  \centering

  \begin{tabular}{cc}

  \includegraphics[width=60mm]{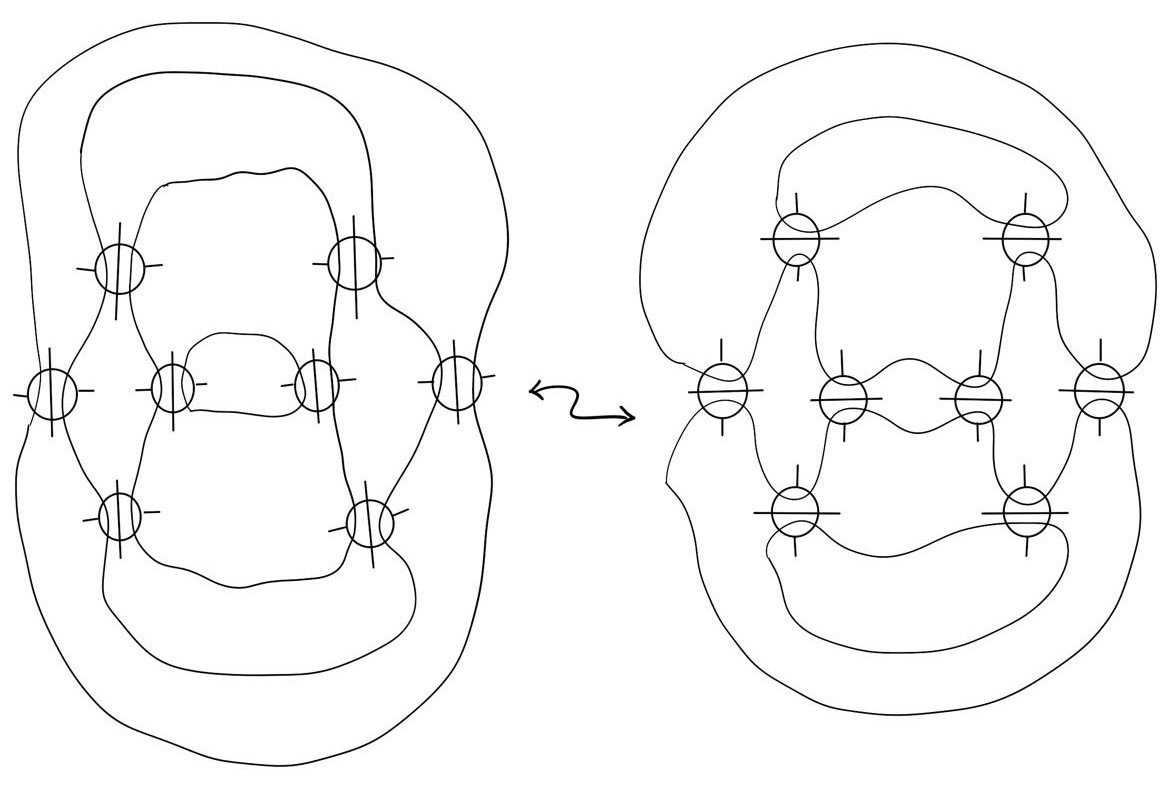}&

    \includegraphics[width=60mm]{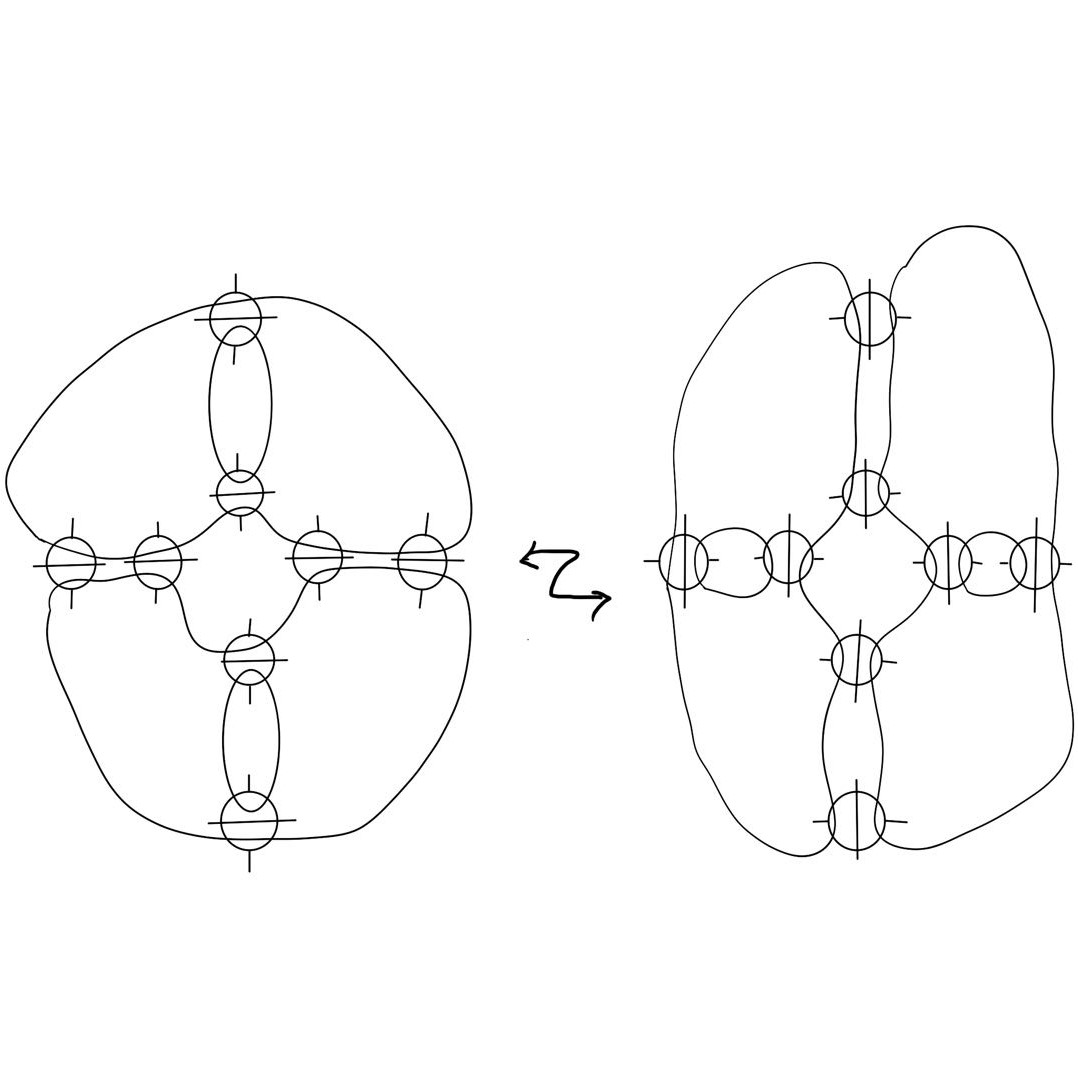}\\

  \end{tabular}
\caption{The configurations of torus when $|R|=8$.}

  \label{R8torus}
\end{figure}

To prove (b), we assume on the pullback graph, $D_N$ is an $N$-gon so that $N$ is the maximum vertex number of the $n$-gons. Note that on $F\cap S^{2}_{+}$ the corresponding loop of $\partial D_N$ divides $S_{+}^{2}$ into two regions, and this loop intersects with a total of $N$ saddles. Then according to Lemma \ref{lower bound}, these two regions contain loops admitting at least $N$ type $I$ arcs which are labeled with $R$. Moreover, $\partial D_N$ admits at least 2 $R$'s. Therefore $N\leq |R|-2$.\par

To prove (c), first we note that Lemma \ref{algorithm} implies the number of regions in the pullback graph is bounded by $|R|$. And by (b) the maximum vertices number of a region in the pullback graph is bounded by $|R|-2$. Therefore, only finitely many configurations exist, due to the finite amount of crossing numbers and type $I$ arcs which are labeled with $R$. Note that each configuration corresponds to an isotopic class of $F$ when it is essential.\par

To prove (d), we calculate the $\chi(F)$ through the pullback graph. $\chi(F)$ is equal to the vertex number $\mathcal{V}$ subtract the edge number $\mathcal{E}$, and plus the region number $\mathcal{F}$ in the pullback graph. Thus, $\chi(F)=\mathcal{V}-\mathcal{E}+\mathcal{F}$=$\sum_{n=2}^{\infty}\frac{n}{4}F_{n}-\sum_{n=2}^{\infty}\frac{n}{2}F_{n}+\sum_{n=2}^{\infty}F_{n}$=$\sum_{n=2}^{\infty}F_{n}-\sum_{n=2}^{\infty}\frac{n}{4}F_{n}$. Here $\sum_{n=2}^{\infty}F_{n}$ equals the region number $\mathcal{F}$, and by Lemma \ref{algorithm} $\mathcal{F}\leq |R|$. $\sum_{n=2}^{\infty}\frac{n}{4}F_{n}$ equals the vertex number $\mathcal{V}$, which is equal to $|S|$. By (b) the maximum possible value of $n$ is $|R|-2$. Therefore we have $\chi(F)=\sum_{n=2}^{|R|-2}F_{n}-\sum_{n=2}^{|R|-2}\frac{n}{4}F_{n}=\sum_{n=2}^{|R|-2}F_{n}-|S| \leq |R|-|S|$. We notice this also implies (c).\par

\end{proof}

\section*{Acknowledgement}

I would like to thank William Menasco for his guidance and many helpful discussions.

\end{document}